\documentclass[12pt]{article}
\usepackage{amsfonts}
\usepackage{CJK}
\usepackage{amsthm}
\usepackage{amsmath}
\usepackage{hyperref}
\usepackage{cleveref}
\usepackage{indentfirst}
\usepackage[top=1em,bottom=1em,left=6em,right=6em,includehead,includefoot]{geometry}
\begin{document}
\begin{CJK}{GBK}{song}

\newtheorem{theorem}{Theorem}
\newtheorem{lemma}{Lemma}
\newtheorem{definition}{Definition}
\newtheorem{remark}{Remark}
\newtheorem{corollery}{Corollery}

\title{\bf The Homology Groups of $G_{n,m}(\mathbb{R})$}
\author{Xu Chen \footnote{{\it Email:} xiaorenwu08@163.com. ChongQing, China }}
\date{}
\maketitle

\begin{abstract}
In this article, we computed the homology groups of real Grassmann manifold $G_{n,m}(\mathbb{R})$ by Witten complex.
\end{abstract}

\section{Introduction}
For a Morse function $f$ on a compact manifold $M$, and choosing a Riemannian metric $g$ on $M$. In [1], Witten introduces a chain complex and the homology groups of $M$ by studing the negative gradient vector field associated to $(f,g)$, and he claimes that these groups are the ordinary homology groups. Salamon gives a more detailed and precise interpretation to Witten's method in [2], we can also see the books [3] and [4]. Here below we give a brief description to Witten's method.

Let $M$ be a n-dimensionial compact smooth manifold and $f$ be a Morse function on $M$. We can choose a suitable Riemannian metric on $M$ such that the negative gradient vector field $-\nabla f$ to be Morse-Smale type(i.e.$f$ is a Morse-Smale function), then the connecting orbits determine the following chain complex.

First choose an orientation of the vector space $E^{u}(p)=T_{p}W^{u}(p)$ for every critical point $p$ of $f$ and denote by $\langle p\rangle$ the pair consisting of the critical point $p$ and the chosen orientation. For $r=1,2,\cdots \dim M$, denote by $C_{r}$ the free group $C_{r}=\bigoplus_{p}\mathbb{Z}\langle p\rangle$ where $p$ runs over all critical points of index $r$. The function $f$ being of Morse-Smale function implies that $W^{u}(p)\cap W^{s}(q)$ consists of finitely many trajectories if ${\rm{ind}}(p)-{\rm{ind}}(q)=1$. In this case one can define an integer $n(p,q)$ by assigning a number $+1$ or $-1$ to every connecting orbit and taking the sum. Let $\varphi(t)$ be such a connecting orbit meaning a solution of $\dot{x}=-\nabla f$ with $\lim_{t\rightarrow -\infty}\varphi(t)=p$ and $\lim_{t\rightarrow +\infty}\varphi(t)=q$. Then $\langle p\rangle$ induces an orientation on the orthogonal complement $E^{u}_{\varphi}(p)$ of $v=\lim_{t\rightarrow-\infty}\mid\dot{\varphi}(t)\mid^{-1}\dot{\varphi}(t)$ in $E^{u}(p)$. Then the negative gradient flow induces an isomorphism between $E^{u}_{\varphi}(p)$ onto $E^{u}(p)$ and we define $n_{\varphi}$ to be $+1$ and $-1$ according to whether this map preserved or reverse the orientation.\\
Define $$n(p,q)=\sum_{\varphi}n_{\varphi}$$ where the sum runs over all orbits of $\dot{x}=-\nabla f$ connecting $p$ and $q$. Then Witten's boundary operator $\partial: C_{r+1}\longrightarrow C_{r}$ of the chain complex is defined by
$$\partial\langle p\rangle=\sum_{q}n(p,q)\langle q\rangle$$ where $q$ runs over all critical points of index $r$.\\

Witten claims that $\partial$ is a boundary operator, i.e.$\partial^{2}=0$, therefore
$$0\rightarrow C_{n}\rightarrow C_{n-1}\rightarrow\cdots\rightarrow C_{1}\rightarrow C_{0}\rightarrow 0$$
is a chain complex. Define for $r=0,1,2,\cdots,n,$
$$H^{W}_{r}(M,\mathbb{Z})=\frac{\ker\{\partial:C_{r}\longrightarrow C_{r-1} \}}{\partial(C_{r+1})}$$
\begin{enumerate}
\item $H^{W}_{*}(M,\mathbb{Z})$ are independent of the choice of the Riemannian metric on $M$;
\item $H^{W}_{*}(M,\mathbb{Z})=H_{*}(M,\mathbb{Z})$, the usual homology groups of $M$.
\end{enumerate}

To compute the homology groups of real Grassmann manifold by Witten complex comes from the work of Feng Hui-tao on $G_{5,2}(\mathbb{R})$(see [5]) for provide a nontrivial example of Witten's method. In [6], Qiao Pei-zhi gives the result on $\mathbb{R}P^{n}$, i.e.$G_{n+1,1}(\mathbb{R})$. In [7], Yang Ying gives the result on $G_{n,2}(\mathbb{R})$. In this article we will give the result on $G_{n,m}(\mathbb{R})$. Computation of the homology groups of $G_{n,m}(\mathbb{C})$ by Witten complex see [4]. We must point out that the comuptation of the homology groups of Grassmann manifold has been known to topologist by use Schubert calculus(see [10] and [11]).

\section{A Morse funtion on the Grassmann manifold $G_{n,m}(\mathbb{R})$}
Let the Grassmann manifold $G_{n,m}(\mathbb{R})$ is the set of all m-dimensional linear subspaces of n-dimensional real vector space $\mathbb{R}^n$.
Set
$$
M=\left\{(x_{\alpha k})=\left(
  \begin{array}{cccc}
    x_{11} & x_{12} & \cdots & x_{1n} \\
    x_{21} & x_{22} & \cdots & x_{2n} \\
    \hdotsfor{4}\\
    x_{m1} & x_{m2} & \cdots & x_{mn} \\
  \end{array}
\right) \Bigg| x_{\alpha k}\in\mathbb{R},\textmd{rank}(x_{\alpha k})=m\right\}
,$$
then the Grassmann manifold $G_{n,m}(\mathbb{R})$ can be defined by $$G_{n,m}(\mathbb{R})=M/GL(m,\mathbb{R}).$$

Let $\pi :M\rightarrow G_{n,m}(\mathbb{R})$ is the standard projection and define

$$
U_{i_{1}i_{2}\cdots i_{m}}=\left\{(x_{\alpha k})\Bigg|x_{\alpha k}\in
M,\left(\begin{array}{cccc}
    x_{1i_{1}} & x_{1i_{2}} & \cdots & x_{1i_{m}} \\
    x_{2i_{1}} & x_{2i_{2}} & \cdots & x_{2i_{m}} \\
    \hdotsfor{4}\\
    x_{mi_{1}} & x_{mi_{2}} & \cdots & x_{mi_{m}} \\
  \end{array}\right)=I_{m}
  \right\},1\leq i_{1}<i_{2}<\cdots<i_{m}\leq n.
$$
where $I_{m}$ is the identity matrix of rank m.

$$\phi _{i_{1}i_{2}\cdots i_{m}}:U_{i_{1}i_{2}\cdots i_{m}}\rightarrow \mathbb{R}^{m(n-m)},$$
$$\phi_{i_{1}i_{2}\cdots i_{m}}(x_{\alpha k})=(x_{1k_1},x_{1k_2},\cdots
,x_{1k_{n-m}},x_{2k_1},x_{2k_2},\cdots ,x_{2k_{n-m}},\cdots,x_{mk_1},x_{mk_2},\cdots
,x_{mk_{n-m}}).$$
where $k_s\neq i_{1},i_{2},\cdots,i_{m};~s=1,2,\cdots ,n-m;~1\leq k_1<\cdots <k_{n-m}\leq n.$

Set
$$\widetilde{U}_{i_{1}i_{2}\cdots i_{m}}=\pi (U_{i_{1}i_{2}\cdots i_{m}}),~\widetilde{\phi }_{i_{1}i_{2}\cdots i_{m}}=\phi_{i_{1}i_{2}\cdots i_{m}}\circ \pi ^{-1}:\widetilde{U}_{i_{1}i_{2}\cdots i_{m}}\rightarrow \mathbb{R}^{m(n-m)}.$$
Then $\{(\widetilde{U}_{i_{1}i_{2}\cdots i_{m}},\widetilde{\phi }_{i_{1}i_{2}\cdots i_{m}})|1\leq i_{1}<i_{2}<\cdots<i_{m}\leq n\}$ are the local coordinate covering of $G_{n,m}(\mathbb{R})$.

Here we will construct the Morse function on $G_{n,m}(\mathbb{R})$ by the same way as in [5],[6],[7].

Let $0<\lambda_{1}<\lambda_{2}<\cdots<\lambda_{n}$ be fixed numbers, for any $(x_{\alpha k})\in M$,
set $$\xi _\alpha =(x_{\alpha 1},x_{\alpha 2},\cdots ,x_{\alpha n}),$$
$$\overline{\xi }_\alpha=(\dfrac{x_{\alpha 1}}{\lambda
_1},\dfrac{x_{\alpha 2}}{\lambda _2},\cdots \dfrac{x_{\alpha
n}}{\lambda _n}),~\alpha =1,2,\cdots,m.$$

$$
\Delta=\det\left(\begin{array}{cccc}
    \langle \xi _1,\xi _1\rangle & \langle \xi _1,\xi _2\rangle & \cdots & \langle \xi _1,\xi _m\rangle \\
    \langle \xi _2,\xi _1\rangle & \langle \xi _2,\xi _2\rangle & \cdots & \langle \xi _2,\xi _m\rangle \\
    \hdotsfor{4}\\
    \langle \xi _m,\xi _1\rangle & \langle \xi _m,\xi _2\rangle & \cdots & \langle \xi _m,\xi _m\rangle \\
  \end{array}\right)
$$

$$
\bar{\Delta}=\det\left(\begin{array}{cccc}
    \langle \overline{\xi}_1,\overline{\xi}_1\rangle & \langle \overline{\xi}_1,\overline{\xi}_2\rangle & \cdots & \langle \overline{\xi}_1,\overline{\xi}_m\rangle \\
    \langle \overline{\xi}_2,\overline{\xi}_1\rangle & \langle \overline{\xi}_2,\overline{\xi}_2\rangle & \cdots & \langle \overline{\xi}_2,\overline{\xi}_m\rangle \\
    \hdotsfor{4}\\
    \langle \overline{\xi}_m,\overline{\xi}_1\rangle & \langle \overline{\xi}_m,\overline{\xi}_2\rangle & \cdots & \langle \overline{\xi}_m,\overline{\xi}_m\rangle \\
  \end{array}\right)
$$
where $$\langle \xi _i,\xi _j\rangle=x_{i1}x_{j1}+x_{i2}x_{j2}+\cdots+x_{im}x_{jm}, i,j=1,2,\cdots,m.$$
we define a function $f$ on $M$ by
$$f\doteq\frac{\Delta}{\bar{\Delta}}$$

\begin{lemma}
The funtion $f$ is defined on the Grassmann manifold $G_{n,m}(\mathbb{R})$.
\end{lemma}
\begin{proof}
Let $A\in GL(m,\mathbb{R})$, because
$$\left(\begin{array}{cccc}
    \langle \xi _1,\xi _1\rangle & \langle \xi _1,\xi _2\rangle & \cdots & \langle \xi _1,\xi _m\rangle \\
    \langle \xi _2,\xi _1\rangle & \langle \xi _2,\xi _2\rangle & \cdots & \langle \xi _2,\xi _m\rangle \\
    \hdotsfor{4}\\
    \langle \xi _m,\xi _1\rangle & \langle \xi _m,\xi _2\rangle & \cdots & \langle \xi _m,\xi _m\rangle \\
  \end{array}\right)=\left(
  \begin{array}{cccc}
    x_{11} & x_{12} & \cdots & x_{1n} \\
    x_{21} & x_{22} & \cdots & x_{2n} \\
    \hdotsfor{4}\\
    x_{m1} & x_{m2} & \cdots & x_{mn} \\
  \end{array}\right)
\left(
  \begin{array}{cccc}
    x_{11} & x_{12} & \cdots & x_{1n} \\
    x_{21} & x_{22} & \cdots & x_{2n} \\
    \hdotsfor{4}\\
    x_{m1} & x_{m2} & \cdots & x_{mn} \\
  \end{array}
\right)^{T}$$
and
$$\left(\begin{array}{cccc}
    \langle \overline{\xi}_1,\overline{\xi}_1\rangle & \langle \overline{\xi}_1,\overline{\xi}_2\rangle & \cdots & \langle \overline{\xi}_1,\overline{\xi}_m\rangle \\
    \langle \overline{\xi}_2,\overline{\xi}_1\rangle & \langle \overline{\xi}_2,\overline{\xi}_2\rangle & \cdots & \langle \overline{\xi}_2,\overline{\xi}_m\rangle \\
    \hdotsfor{4}\\
    \langle \overline{\xi}_m,\overline{\xi}_1\rangle & \langle \overline{\xi}_m,\overline{\xi}_2\rangle & \cdots & \langle \overline{\xi}_m,\overline{\xi}_m\rangle \\
  \end{array}\right)$$
$$=\left(
  \begin{array}{cccc}
    x_{11} & x_{12} & \cdots & x_{1n} \\
    x_{21} & x_{22} & \cdots & x_{2n} \\
    \hdotsfor{4}\\
    x_{m1} & x_{m2} & \cdots & x_{mn} \\
  \end{array}
\right)
\begin{pmatrix}
\frac{1}{\lambda^{2}_{1}}\\
& \frac{1}{\lambda^{2}_{2}} & & \text{{\huge{0}}}\\
& & \ddots \\
& & & \frac{1}{\lambda^{2}_{m-1}} &\\
& \text{{\huge{0}}} & & &\frac{1}{\lambda^{2}_{m}}
\end{pmatrix}
\left(
  \begin{array}{cccc}
    x_{11} & x_{12} & \cdots & x_{1n} \\
    x_{21} & x_{22} & \cdots & x_{2n} \\
    \hdotsfor{4}\\
    x_{m1} & x_{m2} & \cdots & x_{mn} \\
  \end{array}
\right)^{T}$$

where $\left(
  \begin{array}{cccc}
    x_{11} & x_{12} & \cdots & x_{1n} \\
    x_{21} & x_{22} & \cdots & x_{2n} \\
    \hdotsfor{4}\\
    x_{m1} & x_{m2} & \cdots & x_{mn} \\
  \end{array}
\right)^{T}$ is the transposed matrix of $\left(
  \begin{array}{cccc}
    x_{11} & x_{12} & \cdots & x_{1n} \\
    x_{21} & x_{22} & \cdots & x_{2n} \\
    \hdotsfor{4}\\
    x_{m1} & x_{m2} & \cdots & x_{mn} \\
  \end{array}
\right)$.

By computation we have
$$\left(\begin{array}{cccc}
    \langle A\xi_1,A\xi_1\rangle & \langle A\xi_1,A\xi_2\rangle & \cdots & \langle A\xi_1,A\xi_m\rangle \\
    \langle A\xi_2,A\xi_1\rangle & \langle A\xi_2,A\xi_2\rangle & \cdots & \langle A\xi_2,A\xi_m\rangle \\
    \hdotsfor{4}\\
    \langle A\xi_m,A\xi_1\rangle & \langle A\xi_m,A\xi_2\rangle & \cdots & \langle A\xi_m,A\xi_m\rangle \\
  \end{array}\right)=A\left(\begin{array}{cccc}
    \langle \xi _1,\xi _1\rangle & \langle \xi _1,\xi _2\rangle & \cdots & \langle \xi _1,\xi _m\rangle \\
    \langle \xi _2,\xi _1\rangle & \langle \xi _2,\xi _2\rangle & \cdots & \langle \xi _2,\xi _m\rangle \\
    \hdotsfor{4}\\
    \langle \xi _m,\xi _1\rangle & \langle \xi _m,\xi _2\rangle & \cdots & \langle \xi _m,\xi _m\rangle \\
  \end{array}\right)A^{T}$$
and
$$\left(\begin{array}{cccc}
    \langle A\overline{\xi}_1,A\overline{\xi}_1\rangle & \langle A\overline{\xi}_1,A\overline{\xi}_2\rangle & \cdots & \langle A\overline{\xi}_1,A\overline{\xi}_m\rangle \\
    \langle A\overline{\xi}_2,A\overline{\xi}_1\rangle & \langle A\overline{\xi}_2,A\overline{\xi}_2\rangle & \cdots & \langle A\overline{\xi}_2,A\overline{\xi}_m\rangle \\
    \hdotsfor{4}\\
    \langle A\overline{\xi}_m,A\overline{\xi}_1\rangle & \langle A\overline{\xi}_m,A\overline{\xi}_2\rangle & \cdots & \langle A\overline{\xi}_m,A\overline{\xi}_m\rangle \\
  \end{array}\right)=A\left(\begin{array}{cccc}
    \langle \overline{\xi}_1,\overline{\xi}_1\rangle & \langle \overline{\xi}_1,\overline{\xi}_2\rangle & \cdots & \langle \overline{\xi}_1,\overline{\xi}_m\rangle \\
    \langle \overline{\xi}_2,\overline{\xi}_1\rangle & \langle \overline{\xi}_2,\overline{\xi}_2\rangle & \cdots & \langle \overline{\xi}_2,\overline{\xi}_m\rangle \\
    \hdotsfor{4}\\
    \langle \overline{\xi}_m,\overline{\xi}_1\rangle & \langle \overline{\xi}_m,\overline{\xi}_2\rangle & \cdots & \langle \overline{\xi}_m,\overline{\xi}_m\rangle \\
  \end{array}\right)A^{T}$$
where $A^{T}$ is the transposed matrix of $A$. So we have the function $\frac{\Delta}{\bar{\Delta}}$ is invariant under the natural left action of the group $GL(m,\mathbb{R})$, then we get $f$ is defined on the Grassmann manifold $G_{n,m}(\mathbb{R})$.

\end{proof}

\begin{lemma}
The funtion $f$ is the Morse function on the Grassmann manifold $G_{n,m}(\mathbb{R})$, has and only has one critical point $O_{i_{1}i_{2}\cdots i_{m}}=\widetilde{\phi}^{-1}_{i_{1}i_{2}\cdots i_{m}}(\underbrace{0,0,\cdots,0}_{m(n-m)})$ on each $\widetilde{U}_{i_{1}i_{2}\cdots i_{m}}$, $1\leq i_{1}<i_{2}<\cdots<i_{m}\leq n$.
\end{lemma}
\begin{proof}
By the definition $$f=\frac{\Delta}{\bar{\Delta}}$$
so on the local coordinate systems $\widetilde{U}_{i_{1}i_{2}\cdots i_{m}}$ we have
\begin{equation}\frac{\partial f}{\partial
x_{ij}}=\frac{\frac{\partial \Delta}{\partial
x_{ij}}\cdot\bar{\Delta}-\Delta\cdot\frac{\partial \bar{\Delta}}{\partial
x_{ij}}}{\bar{\Delta}\cdot\bar{\Delta}}.
\end{equation}
where $i=1,2,\cdots,m$; $j=k_{s}\neq i_{1},i_{2},\cdots,i_{m}$.
By computation we have
$$\frac{\partial \Delta}{\partial
x_{ij}}=2\det\left(\begin{array}{cccccccc}
    \langle \xi _1,\xi _1\rangle & \langle \xi _1,\xi _2\rangle & \cdots & \langle \xi _1,\xi _{i-1}\rangle & x_{1j} & \langle \xi _1,\xi _{i+1}\rangle & \cdots & \langle \xi _1,\xi _m\rangle\\
    \langle \xi _2,\xi _1\rangle & \langle \xi _2,\xi _2\rangle & \cdots & \langle \xi _2,\xi _{i-1}\rangle & x_{2j} & \langle \xi _2,\xi _{i+1}\rangle & \cdots & \langle \xi _2,\xi _m\rangle\\
    \hdotsfor{8}\\
    \langle \xi _m,\xi _1\rangle & \langle \xi _m,\xi _2\rangle & \cdots & \langle \xi _m,\xi _{i-1}\rangle & x_{mj} & \langle \xi _m,\xi _{i+1}\rangle & \cdots & \langle \xi _m,\xi _m\rangle\\
  \end{array}\right)$$

$$\frac{\partial \bar{\Delta}}{\partial
x_{ij}}=2\det\left(\begin{array}{cccccccc}
    \langle \overline{\xi} _1,\overline{\xi} _1\rangle & \langle \overline{\xi} _1,\overline{\xi} _2\rangle & \cdots & \langle \overline{\xi} _1,\overline{\xi} _{i-1}\rangle & \frac{x_{1j}}{\lambda^{2}_{j}} & \langle \overline{\xi} _1,\overline{\xi} _{i+1}\rangle & \cdots & \langle \overline{\xi} _1,\overline{\xi} _m\rangle\\
    \langle \overline{\xi} _2,\overline{\xi} _1\rangle & \langle \overline{\xi} _2,\overline{\xi} _2\rangle & \cdots & \langle \overline{\xi} _2,\overline{\xi} _{i-1}\rangle & \frac{x_{2j}}{\lambda^{2}_{j}} & \langle \overline{\xi} _2,\overline{\xi} _{i+1}\rangle & \cdots & \langle \overline{\xi} _2,\overline{\xi} _m\rangle\\
    \hdotsfor{8}\\
    \langle \overline{\xi} _m,\overline{\xi} _1\rangle & \langle \overline{\xi} _m,\overline{\xi} _2\rangle & \cdots & \langle \overline{\xi} _m,\overline{\xi} _{i-1}\rangle & \frac{x_{mj}}{\lambda^{2}_{j}} & \langle \overline{\xi} _m,\overline{\xi} _{i+1}\rangle & \cdots & \langle \overline{\xi} _m,\overline{\xi} _m\rangle\\
  \end{array}\right)$$
then
$$\frac{\partial \Delta}{\partial
x_{ij}}\cdot\bar{\Delta}-\Delta\cdot\frac{\partial \bar{\Delta}}{\partial
x_{ij}}=$$
\begin{equation}2(x_{1j}A_{1i}+x_{2j}A_{2i}+\cdots+x_{mj}A_{mi})\cdot\bar{\Delta}-2(\frac{x_{1j}}{\lambda^{2}_{j}}\overline{A}_{1i}+\frac{x_{2j}}{\lambda^{2}_{j}}\overline{A}_{2i}+\cdots+\frac{x_{mj}}{\lambda^{2}_{j}}\overline{A}_{mi})\cdot\Delta\end{equation}
where $A_{ij}$ is $(i,j)$ cofactor of

$$\left(\begin{array}{cccc}
    \langle \xi _1,\xi _1\rangle & \langle \xi _1,\xi _2\rangle & \cdots & \langle \xi _1,\xi _m\rangle \\
    \langle \xi _2,\xi _1\rangle & \langle \xi _2,\xi _2\rangle & \cdots & \langle \xi _2,\xi _m\rangle \\
    \hdotsfor{4}\\
    \langle \xi _m,\xi _1\rangle & \langle \xi _m,\xi _2\rangle & \cdots & \langle \xi _m,\xi _m\rangle \\
  \end{array}\right),$$
$\overline{A}_{ij}$ is $(i,j)$ cofactor of
$$\left(\begin{array}{cccc}
    \langle \overline{\xi}_1,\overline{\xi}_1\rangle & \langle \overline{\xi}_1,\overline{\xi}_2\rangle & \cdots & \langle \overline{\xi}_1,\overline{\xi}_m\rangle \\
    \langle \overline{\xi}_2,\overline{\xi}_1\rangle & \langle \overline{\xi}_2,\overline{\xi}_2\rangle & \cdots & \langle \overline{\xi}_2,\overline{\xi}_m\rangle \\
    \hdotsfor{4}\\
    \langle \overline{\xi}_m,\overline{\xi}_1\rangle & \langle \overline{\xi}_m,\overline{\xi}_2\rangle & \cdots & \langle \overline{\xi}_m,\overline{\xi}_m\rangle \\
  \end{array}\right)$$
so we have $O_{i_{1}i_{2}\cdots i_{m}}=\widetilde{\phi}^{-1}_{i_{1}i_{2}\cdots i_{m}}(\underbrace{0,0,\cdots,0}_{m(n-m)})$ is a critical point of $f$. In the following we can proof that $O_{i_{1}i_{2}\cdots i_{m}}$ is the only one critical point on each $\widetilde{U}_{i_{1}i_{2}\cdots i_{m}}$.

If there exists another critical point $p\neq O_{i_{1}i_{2}\cdots i_{m}}$, then for some $i,j$ we have $x_{ij}\neq0$, and according to the definition of critical point $\frac{\partial f}{\partial x_{ij}}(p)=0$, then $\frac{\partial \Delta}{\partial
x_{ij}}\cdot\bar{\Delta}-\Delta\cdot\frac{\partial \bar{\Delta}}{\partial
x_{ij}}=0$, so
$$(x_{1j}A_{1i}+x_{2j}A_{2i}+\cdots+x_{mj}A_{mi})\cdot\bar{\Delta}-(\frac{x_{1j}}{\lambda^{2}_{j}}\overline{A}_{1i}+\frac{x_{2j}}{\lambda^{2}_{j}}\overline{A}_{2i}+\cdots+\frac{x_{mj}}{\lambda^{2}_{j}}\overline{A}_{mi})\cdot\Delta=0$$
If we get $i$ from 1 to $m$, by them we can get a linear system of equations, because $x_{ij}\neq0$ so we get
$$\det\left(
  \begin{array}{cccc}
    A_{11}\cdot\bar{\Delta}-\frac{1}{\lambda^{2}_{j}}\overline{A}_{11}\cdot\Delta & A_{21}\cdot\bar{\Delta}-\frac{1}{\lambda^{2}_{j}}\overline{A}_{21}\cdot\Delta & \cdots & A_{m1}\cdot\bar{\Delta}-\frac{1}{\lambda^{2}_{j}}\overline{A}_{m1}\cdot\Delta \\
    A_{12}\cdot\bar{\Delta}-\frac{1}{\lambda^{2}_{j}}\overline{A}_{12}\cdot\Delta & A_{22}\cdot\bar{\Delta}-\frac{1}{\lambda^{2}_{j}}\overline{A}_{22}\cdot\Delta & \cdots & A_{m2}\cdot\bar{\Delta}-\frac{1}{\lambda^{2}_{j}}\overline{A}_{m2}\cdot\Delta \\
    \hdotsfor{4}\\
    A_{1m}\cdot\bar{\Delta}-\frac{1}{\lambda^{2}_{j}}\overline{A}_{1m}\cdot\Delta & A_{2m}\cdot\bar{\Delta}-\frac{1}{\lambda^{2}_{j}}\overline{A}_{2m}\cdot\Delta & \cdots & A_{mm}\cdot\bar{\Delta}-\frac{1}{\lambda^{2}_{j}}\overline{A}_{mm}\cdot\Delta \\
  \end{array}\right)=0$$
we denote the above matrix by $B$ , because
$$\left(\begin{array}{cccc}
    \langle \xi _1,\xi _1\rangle & \langle \xi _1,\xi _2\rangle & \cdots & \langle \xi _1,\xi _m\rangle \\
    \langle \xi _2,\xi _1\rangle & \langle \xi _2,\xi _2\rangle & \cdots & \langle \xi _2,\xi _m\rangle \\
    \hdotsfor{4}\\
    \langle \xi _m,\xi _1\rangle & \langle \xi _m,\xi _2\rangle & \cdots & \langle \xi _m,\xi _m\rangle \\
  \end{array}\right)B$$
$$=\left(\begin{array}{cccc}
    \langle\overline{\xi} _1,\overline{\xi} _1\rangle-\frac{\langle \xi _1,\xi _1\rangle}{\lambda^{2}_{j}} & \langle\overline{\xi} _1,\overline{\xi} _2\rangle-\frac{\langle \xi _1,\xi _2\rangle}{\lambda^{2}_{j}} & \cdots & \langle\overline{\xi} _1,\overline{\xi} _m\rangle-\frac{\langle \xi _1,\xi _m\rangle}{\lambda^{2}_{j}} \\
    \langle\overline{\xi}_2,\overline{\xi}_1\rangle-\frac{\langle \xi _2,\xi _1\rangle}{\lambda^{2}_{j}} & \langle\overline{\xi} _2,\overline{\xi} _2\rangle-\frac{\langle \xi _2,\xi _2\rangle}{\lambda^{2}_{j}} & \cdots & \langle\overline{\xi} _2,\overline{\xi} _m\rangle-\frac{\langle \xi _2,\xi _m\rangle}{\lambda^{2}_{j}} \\
    \hdotsfor{4}\\
    \langle\overline{\xi} _m,\overline{\xi} _1\rangle-\frac{\langle \xi _m,\xi _1\rangle}{\lambda^{2}_{j}} & \langle\overline{\xi} _m,\overline{\xi} _2\rangle-\frac{\langle \xi _m,\xi _2\rangle}{\lambda^{2}_{j}} & \cdots & \langle\overline{\xi} _m,\overline{\xi} _m\rangle-\frac{\langle \xi _m,\xi _m\rangle}{\lambda^{2}_{j}} \\
  \end{array}\right)C$$
where
$$C= \left(
  \begin{array}{cccc}
    \overline{A}_{11}\cdot\Delta & \overline{A}_{21}\cdot\Delta & \cdots & \overline{A}_{m1}\cdot\Delta \\
    \overline{A}_{12}\cdot\Delta & \overline{A}_{22}\cdot\Delta & \cdots & \overline{A}_{m2}\cdot\Delta \\
    \hdotsfor{4}\\
    \overline{A}_{1m}\cdot\Delta & \overline{A}_{2m}\cdot\Delta & \cdots & \overline{A}_{mm}\cdot\Delta \\
  \end{array}\right)=\Delta\cdot\left(
  \begin{array}{cccc}
    \overline{A}_{11} & \overline{A}_{21} & \cdots & \overline{A}_{m1} \\
    \overline{A}_{12} & \overline{A}_{22} & \cdots & \overline{A}_{m2} \\
    \hdotsfor{4}\\
    \overline{A}_{1m} & \overline{A}_{2m} & \cdots & \overline{A}_{mm} \\
  \end{array}\right).$$
by computation we have
$$\left(\begin{array}{cccc}
    \langle\overline{\xi} _1,\overline{\xi} _1\rangle-\frac{\langle \xi _1,\xi _1\rangle}{\lambda^{2}_{j}} & \langle\overline{\xi} _1,\overline{\xi} _2\rangle-\frac{\langle \xi _1,\xi _2\rangle}{\lambda^{2}_{j}} & \cdots & \langle\overline{\xi} _1,\overline{\xi} _m\rangle-\frac{\langle \xi _1,\xi _m\rangle}{\lambda^{2}_{j}} \\
    \langle\overline{\xi}_2,\overline{\xi}_1\rangle-\frac{\langle \xi _2,\xi _1\rangle}{\lambda^{2}_{j}} & \langle\overline{\xi} _2,\overline{\xi} _2\rangle-\frac{\langle \xi _2,\xi _2\rangle}{\lambda^{2}_{j}} & \cdots & \langle\overline{\xi} _2,\overline{\xi} _m\rangle-\frac{\langle \xi _2,\xi _m\rangle}{\lambda^{2}_{j}} \\
    \hdotsfor{4}\\
    \langle\overline{\xi} _m,\overline{\xi} _1\rangle-\frac{\langle \xi _m,\xi _1\rangle}{\lambda^{2}_{j}} & \langle\overline{\xi} _m,\overline{\xi} _2\rangle-\frac{\langle \xi _m,\xi _2\rangle}{\lambda^{2}_{j}} & \cdots & \langle\overline{\xi} _m,\overline{\xi} _m\rangle-\frac{\langle \xi _m,\xi _m\rangle}{\lambda^{2}_{j}} \\
  \end{array}\right)$$

$$=X
\begin{pmatrix}
\frac{1}{\lambda^{2}_{1}}-\frac{1}{\lambda^{2}_{j}}\\
& \frac{1}{\lambda^{2}_{2}}-\frac{1}{\lambda^{2}_{j}} & & & & \text{{\huge{0}}}\\
& & \ddots \\
& & & \frac{1}{\lambda^{2}_{j-1}}-\frac{1}{\lambda^{2}_{j}} & & & &\\
& & & & \frac{1}{\lambda^{2}_{j+1}}-\frac{1}{\lambda^{2}_{j}} & & &\\
& & & & & \ddots \\
& & \text{{\huge{0}}} & & & & \frac{1}{\lambda^{2}_{n-1}}-\frac{1}{\lambda^{2}_{j}}\\
& & & & & & &\frac{1}{\lambda^{2}_{n}}-\frac{1}{\lambda^{2}_{j}}
\end{pmatrix}X^{T}
$$
where
$$X=\begin{pmatrix}
    x_{11} & x_{12} & \cdots & x_{1(j-1)}& x_{1(j+1)} & \cdots & x_{1n} \\
    x_{21} & x_{22} & \cdots & x_{2(j-1)}& x_{2(j+1)} & \cdots & x_{2n} \\
    \hdotsfor{7}\\
    x_{m1} & x_{m2} & \cdots & x_{m(j-1)}& x_{m(j+1)} & \cdots & x_{mn} \\
  \end{pmatrix}$$
$X^{T}$ is the transposed matrix of $X$.
Because $\det B=0$, so we have

$$\det\left(\begin{array}{cccc}
    \langle\overline{\xi} _1,\overline{\xi} _1\rangle-\frac{\langle \xi _1,\xi _1\rangle}{\lambda^{2}_{j}} & \langle\overline{\xi} _1,\overline{\xi} _2\rangle-\frac{\langle \xi _1,\xi _2\rangle}{\lambda^{2}_{j}} & \cdots & \langle\overline{\xi} _1,\overline{\xi} _m\rangle-\frac{\langle \xi _1,\xi _m\rangle}{\lambda^{2}_{j}} \\
    \langle\overline{\xi}_2,\overline{\xi}_1\rangle-\frac{\langle \xi _2,\xi _1\rangle}{\lambda^{2}_{j}} & \langle\overline{\xi} _2,\overline{\xi} _2\rangle-\frac{\langle \xi _2,\xi _2\rangle}{\lambda^{2}_{j}} & \cdots & \langle\overline{\xi} _2,\overline{\xi} _m\rangle-\frac{\langle \xi _2,\xi _m\rangle}{\lambda^{2}_{j}} \\
    \hdotsfor{4}\\
    \langle\overline{\xi} _m,\overline{\xi} _1\rangle-\frac{\langle \xi _m,\xi _1\rangle}{\lambda^{2}_{j}} & \langle\overline{\xi} _m,\overline{\xi} _2\rangle-\frac{\langle \xi _m,\xi _2\rangle}{\lambda^{2}_{j}} & \cdots & \langle\overline{\xi} _m,\overline{\xi} _m\rangle-\frac{\langle \xi _m,\xi _m\rangle}{\lambda^{2}_{j}} \\
  \end{array}\right)=0$$
because it is independent of the value of $x_{kl}$ ($k=1,2,\cdots,m$; $l=1,2,\cdots,j-1,j+1,\cdots,n.$), so we must have
$\frac{1}{\lambda^{2}_{l}}-\frac{1}{\lambda^{2}_{j}}=0$ for some $l$, then we get $\lambda_{l}=\lambda_{j}$, which contradicts to $\lambda_{l}\neq\lambda_{j},(l\neq j)$. So we get $x_{ij}=0$, there non-exists another critical point $p\neq O_{i_{1}i_{2}\cdots i_{m}}$ on $\widetilde{U}_{i_{1}i_{2}\cdots i_{m}}$.

\end{proof}

\begin{lemma}
The critical point $O_{i_{1}i_{2}\cdots i_{m}}$ on each $\widetilde{U}_{i_{1}i_{2}\cdots i_{m}}$($1\leq i_{1}<i_{2}<\cdots<i_{m}\leq n$) is non-degenerate, and $\mathrm{ind}(O_{i_{1}i_{2}\cdots i_{m}})=i_{1}+i_{2}+\cdots+i_{m}-\frac{1}{2}m(m+1).$
\end{lemma}
\begin{proof}
On $\widetilde{U}_{i_{1}i_{2}\cdots i_{m}}$ we have
$$\frac{\partial^{2} f}{\partial x_{kl} \partial x_{ij}}=\frac{\partial }{\partial x_{kl}}\left(\frac{\frac{\partial \Delta}{\partial
x_{ij}}\cdot\bar{\Delta}-\Delta\cdot\frac{\partial \bar{\Delta}}{\partial
x_{ij}}}{\bar{\Delta}\cdot\bar{\Delta}}\right)$$
$$=\frac{1}{\bar{\Delta}^{4}}\left(\frac{\partial }{\partial x_{kl}}\left(\frac{\partial \Delta}{\partial
x_{ij}}\cdot\bar{\Delta}-\Delta\cdot\frac{\partial \bar{\Delta}}{\partial
x_{ij}}\right)\cdot\bar{\Delta}^{2}-\left(\frac{\partial \Delta}{\partial
x_{ij}}\cdot\bar{\Delta}-\Delta\cdot\frac{\partial \bar{\Delta}}{\partial
x_{ij}}\right)\frac{\partial }{\partial x_{kl}}\bar{\Delta}^{2}\right)$$

$$=\frac{1}{\bar{\Delta}^{4}}\left( \frac{\partial^{2}\Delta}{\partial x_{kl}\partial x_{ij}}\cdot\bar{\Delta}^{3}-\frac{\partial \Delta}{\partial x_{ij}}\frac{\partial \bar{\Delta}}{\partial x_{kl}}\cdot\bar{\Delta}^{2}-\frac{\partial \Delta}{\partial x_{kl}}\frac{\partial \bar{\Delta}}{\partial x_{ij}}\cdot\bar{\Delta}^{2}-\Delta\frac{\partial^{2}\bar{\Delta}}{\partial x_{kl}\partial x_{ij}}\bar{\Delta}^{2}-2\Delta\frac{\partial \bar{\Delta}}{\partial x_{ij}}\frac{\partial \bar{\Delta}}{\partial x_{kl}}\cdot\bar{\Delta}\right)$$
When $k<i$, $l<j$, then
$$\frac{\partial^{2}\Delta}{\partial x_{kl}\partial x_{ij}}=$$

$$2\det\left(\begin{array}{ccccccccc}
    \langle \xi _1,\xi _1\rangle & \cdots & \langle \xi _1,\xi _k\rangle & \cdots & \langle \xi _1,\xi _{i-1}\rangle & x_{1j} & \langle \xi _1,\xi _{i+1}\rangle & \cdots & \langle \xi _1,\xi _m\rangle \\
    \hdotsfor{9}\\
    \langle \xi _{k-1},\xi _1\rangle & \cdots & \langle \xi _{k-1},\xi _k\rangle & \cdots & \langle \xi _{k-1},\xi _{i-1}\rangle & x_{(k-1)j} & \langle \xi _{k-1},\xi _{i+1}\rangle & \cdots & \langle \xi _{k-1},\xi _m\rangle \\
    x_{1l} & \cdots & x_{kl} & \cdots & x_{(i-1)l} & 0 & x_{(i+1)l} & \cdots & x_{ml}  \\
    \langle \xi _{k+1},\xi _1\rangle & \cdots & \langle \xi _{k+1},\xi _k\rangle & \cdots & \langle \xi _{k+1},\xi _{i-1}\rangle & x_{(k+1)j} & \langle \xi _{k+1},\xi _{i+1}\rangle & \cdots & \langle \xi _{k+1},\xi _m\rangle\\
    \hdotsfor{9}\\
    \langle \xi _m,\xi _1\rangle & \cdots & \langle \xi _{m},\xi _k\rangle & \cdots & \langle \xi _m,\xi _{i-1}\rangle & x_{mj} & \langle \xi _m,\xi _{i+1}\rangle & \cdots & \langle \xi _m,\xi _m\rangle\\
  \end{array}\right)$$

$$+2\det\left(\begin{array}{ccccccccc}
    \langle \xi _1,\xi _1\rangle & \cdots & \langle \xi _1,\xi _{k-1}\rangle & x_{1l} & \cdots & \langle \xi_1,\xi _{i-1}\rangle & x_{1j} &\cdots & \langle \xi _1,\xi _m\rangle \\
    \hdotsfor{9}\\
    \langle \xi _{k-1},\xi _1\rangle & \cdots & \langle \xi _{k-1},\xi _{k-1}\rangle & x_{(k-1)l} & \cdots & \langle \xi _{k-1},\xi _{i-1}\rangle & x_{(k-1)j} & \cdots & \langle \xi _{k-1},\xi _m\rangle \\
    \langle \xi _{k},\xi _{1}\rangle & \cdots & \langle \xi _{k},\xi _{1}\rangle & x_{kl} & \cdots & \langle \xi _{k},\xi _{i-1}\rangle & x_{kj} & \cdots & \langle \xi _{k},\xi _{m}\rangle  \\
    \langle \xi _{k+1},\xi _1\rangle & \cdots & \langle \xi _{k+1},\xi _{k-1}\rangle & x_{(k+1)l} & \cdots & \langle \xi _{k+1},\xi _{i-1}\rangle & x_{(k+1)j} & \cdots & \langle \xi _{k+1},\xi _m\rangle \\
    \hdotsfor{9} \\
    \langle \xi _m,\xi _1\rangle & \cdots & \langle \xi _{m},\xi _{k_1}\rangle & x_{ml} & \cdots & \langle \xi _m,\xi _{i-1}\rangle & x_{mj} & \cdots & \langle \xi _m,\xi _m\rangle \\
 \end{array}\right).$$
so we have $\frac{\partial^{2}\Delta}{\partial x_{kl}\partial x_{ij}}(O_{i_{1}i_{2}\cdots i_{m}})=0$. When $k>i$,$l<j$ and $k\neq i$,$l>j$ the result is the same. By the same way we can get $\frac{\partial^{2}\bar{\Delta}}{\partial x_{kl}\partial x_{ij}}(O_{i_{1}i_{2}\cdots i_{m}})=0$.

When $k\neq i$, $l=j$, then
$$\frac{\partial^{2}\Delta}{\partial x_{kj}\partial x_{ij}}=$$

$$2\det\left(\begin{array}{ccccccccc}
    \langle \xi _1,\xi _1\rangle & \cdots & \langle \xi _1,\xi _k\rangle & \cdots & \langle \xi _1,\xi _{i-1}\rangle & x_{1j} & \langle \xi _1,\xi _{i+1}\rangle & \cdots & \langle \xi _1,\xi _m\rangle \\
    \hdotsfor{9}\\
    \langle \xi _{k-1},\xi _1\rangle & \cdots & \langle \xi _{k-1},\xi _k\rangle & \cdots & \langle \xi _{k-1},\xi _{i-1}\rangle & x_{(k-1)j} & \langle \xi _{k-1},\xi _{i+1}\rangle & \cdots & \langle \xi _{k-1},\xi _m\rangle \\
    x_{1j} & \cdots & x_{kj} & \cdots & x_{(i-1)j} & 1 & x_{(i+1)j} & \cdots & x_{mj}  \\
    \langle \xi _{k+1},\xi _1\rangle & \cdots & \langle \xi _{k+1},\xi _k\rangle & \cdots & \langle \xi _{k+1},\xi _{i-1}\rangle & x_{(k+1)j} & \langle \xi _{k+1},\xi _{i+1}\rangle & \cdots & \langle \xi _{k+1},\xi _m\rangle\\
    \hdotsfor{9}\\
    \langle \xi _m,\xi _1\rangle & \cdots & \langle \xi _{m},\xi _k\rangle & \cdots & \langle \xi _m,\xi _{i-1}\rangle & x_{mj} & \langle \xi _m,\xi _{i+1}\rangle & \cdots & \langle \xi _m,\xi _m\rangle\\
  \end{array}\right)$$
$$=2A_{ki}$$
because $A_{ki}(O_{i_{1}i_{2}\cdots i_{m}})=0$, so we get $\frac{\partial^{2}\Delta}{\partial x_{kj}\partial x_{ij}}(O_{i_{1}i_{2}\cdots i_{m}})=0$, by the same way we can get $\frac{\partial^{2}\bar{\Delta}}{\partial x_{kj}\partial x_{ij}}(O_{i_{1}i_{2}\cdots i_{m}})=0$.

When $k=i$, $l=j$, then
$$\frac{\partial^{2}\Delta}{\partial x_{ij}\partial x_{ij}}=
2\det\left(\begin{array}{ccccccc}
    \langle \xi _1,\xi _1\rangle & \cdots & \langle \xi _1,\xi _{i-1}\rangle & x_{1j} & \langle \xi _1,\xi _{i+1}\rangle & \cdots & \langle \xi _1,\xi _m\rangle \\
    \hdotsfor{7}\\
    \langle \xi _{i-1},\xi _1\rangle & \cdots & \langle \xi _{i-1},\xi _{i-1}\rangle & x_{(i-1)j} & \langle \xi _{i-1},\xi _{i+1}\rangle & \cdots & \langle \xi _{i-1},\xi _m\rangle \\
    x_{1j} & \cdots & x_{(i-1)j} & 1 & x_{(i+1)j} & \cdots & x_{mj}  \\
    \langle \xi _{i+1},\xi _1\rangle & \cdots & \langle \xi _{i+1},\xi _{i-1}\rangle & x_{(i+1)j} & \langle \xi _{i+1},\xi _{i+1}\rangle & \cdots & \langle \xi _{i+1},\xi _m\rangle\\
    \hdotsfor{7}\\
    \langle \xi _m,\xi _1\rangle &\cdots & \langle \xi _m,\xi _{i-1}\rangle & x_{mj} & \langle \xi _m,\xi _{i+1}\rangle & \cdots & \langle \xi _m,\xi _m\rangle\\
  \end{array}\right).$$

$$\frac{\partial^{2}\bar{\Delta}}{\partial x_{ij}\partial x_{ij}}=
2\det\left(\begin{array}{ccccccc}
    \langle \overline{\xi} _1,\overline{\xi} _1\rangle & \cdots & \langle \overline{\xi} _1,\overline{\xi} _{i-1}\rangle & \frac{x_{1j}}{\lambda^{2}_{j}} & \langle \overline{\xi} _1,\overline{\xi} _{i+1}\rangle & \cdots & \langle \overline{\xi} _1,\overline{\xi} _m\rangle \\
    \hdotsfor{7}\\
    \langle \overline{\xi} _{i-1},\overline{\xi} _1\rangle & \cdots & \langle \overline{\xi} _{i-1},\overline{\xi} _{i-1}\rangle & \frac{x_{(i-1)j}}{\lambda^{2}_{j}} & \langle \overline{\xi} _{i-1},\overline{\xi} _{i+1}\rangle & \cdots & \langle \overline{\xi} _{i-1},\overline{\xi} _m\rangle \\
    \frac{x_{1j}}{\lambda^{2}_{j}} & \cdots & \frac{x_{(i-1)j}}{\lambda^{2}_{j}} & \frac{1}{\lambda^{2}_{j}} & \frac{x_{(i+1)j}}{\lambda^{2}_{j}} & \cdots & \frac{x_{mj}}{\lambda^{2}_{j}}  \\
    \langle \overline{\xi} _{i+1},\overline{\xi} _1\rangle & \cdots & \langle \overline{\xi} _{i+1},\overline{\xi} _{i-1}\rangle & \frac{x_{(i+1)j}}{\lambda^{2}_{j}} & \langle \overline{\xi} _{i+1},\overline{\xi} _{i+1}\rangle & \cdots & \langle \overline{\xi} _{i+1},\overline{\xi} _m\rangle\\
    \hdotsfor{7}\\
    \langle \overline{\xi} _m,\overline{\xi} _1\rangle &\cdots & \langle \overline{\xi} _m,\overline{\xi} _{i-1}\rangle & \frac{x_{mj}}{\lambda^{2}_{j}} & \langle \overline{\xi} _m,\overline{\xi} _{i+1}\rangle & \cdots & \langle \overline{\xi} _m,\overline{\xi} _m\rangle\\
  \end{array}\right).$$

so $$\frac{\partial^{2}\Delta}{\partial x_{ij}\partial x_{ij}}(O_{i_{1}i_{2}\cdots i_{m}})=2$$
$$\frac{\partial^{2}\bar{\Delta}}{\partial x_{ij}\partial x_{ij}}(O_{i_{1}i_{2}\cdots i_{m}})=2\frac{1}{\lambda^{2}_{i_{1}}\lambda^{2}_{i_{2}}\cdots\lambda^{2}_{i_{k-1}}\lambda^{2}_{j}\lambda^{2}_{i_{k+1}}\cdots\lambda^{2}_{i_{m}}}$$
then by computation we have
$$\frac{\partial^{2} f}{\partial x_{ij} \partial x_{ij}}(O_{i_{1}i_{2}\cdots i_{m}})=2\lambda^{2}_{i_{1}}\lambda^{2}_{i_{2}}\cdots\lambda^{2}_{i_{m}}\left(1-\frac{\lambda^{2}_{i_{k}}}{\lambda^{2}_{j}}\right),$$ here $j=k_{s}$.

So at the critical point $O_{i_{1}i_{2}\cdots i_{m}}$ of $f$ on each $\widetilde{U}_{i_{1}i_{2}\cdots i_{m}}$, we have the Hessian matrix $H_{O_{i_{1}i_{2}\cdots i_{m}}}(f)$
$$H_{O_{i_{1}i_{2}\cdots i_{m}}}(f)=$$
$$2\lambda^{2}_{i_{1}}\lambda^{2}_{i_{2}}\cdots\lambda^{2}_{i_{m}}\mathrm{diag}\left(1-\frac{\lambda^{2}_{i_{1}}}{\lambda^{2}_{k_{1}}},\cdots,1-\frac{\lambda^{2}_{i_{1}}}{\lambda^{2}_{k_{n-m}}},
\cdots,1-\frac{\lambda^{2}_{i_{m}}}{\lambda^{2}_{k_{1}}},\cdots,1-\frac{\lambda^{2}_{i_{m}}}{\lambda^{2}_{k_{n-m}}}\right)$$
then the critical point $O_{i_{1}i_{2}\cdots i_{m}}$ is non-degenerate. By the definition of Morse index of $f$, we have $\mathrm{ind}(O_{i_{1}i_{2}\cdots i_{m}})=(i_{1}-1)+(i_{2}-2)+\cdots+(i_{m}-m)=i_{1}+i_{2}+\cdots+i_{m}-\frac{1}{2}m(m+1)$.

\end{proof}

\section{Riemannian metric on the Grassmann manifold $G_{n,m}(\mathbb{R})$}

Lu Qi-keng introduce a Riemannian metric $g$ on the Grassmann manifold $G_{n,m}(\mathbb{R})$ in [8]. It has the following form
$$g={\rm tr}[(I+ZZ^T)^{-1}dZ(I+Z^TZ)^{-1}dZ^T]$$
where $I$ is the identity matrix,
$$Z=
\begin{pmatrix}
    x_{1k_{1}} & x_{1k_{2}} & \cdots & x_{1k_{n-m}} \\
    x_{2k_{1}} & x_{2k_{2}} & \cdots & x_{2k_{n-m}} \\
    \hdotsfor{4}\\
    x_{mk_{1}} & x_{mk_{2}} & \cdots & x_{mk_{n-m}} \\
\end{pmatrix},$$
$Z^{T}$ is the transposed matrix of $Z$.

we have
$$I+ZZ^T=\left(\begin{array}{cccc}
    \langle \xi _1,\xi _1\rangle & \langle \xi _1,\xi _2\rangle & \cdots & \langle \xi _1,\xi _m\rangle \\
    \langle \xi _2,\xi _1\rangle & \langle \xi _2,\xi _2\rangle & \cdots & \langle \xi _2,\xi _m\rangle \\
    \hdotsfor{4}\\
    \langle \xi _m,\xi _1\rangle & \langle \xi _m,\xi _2\rangle & \cdots & \langle \xi _m,\xi _m\rangle \\
  \end{array}\right)$$

$$(I+ZZ^T)^{-1}=\frac{1}{\Delta}\left(\begin{array}{cccc}
    A_{11} & A_{21} & \cdots & A_{m1} \\
    A_{12} & A_{22} & \cdots & A_{m2} \\
    \hdotsfor{4}\\
    A_{1m} & A_{2m} & \cdots & A_{mm} \\
  \end{array}\right)$$

Set $$\begin{cases}
   \eta_{k_1}=(x_{1k_1},x_{2k_1},\cdots,x_{mk_1},1,0,\cdots,0)\\
   \eta_{k_2}=(x_{1k_2},x_{2k_2},\cdots,x_{mk_2},0,1,\cdots,0)\\
   \cdots\\
   \eta_{k_{n-m}}=(x_{1k_{n-m}},x_{2k_{n-m}},\cdots,x_{mk_{n-m}},0,0,\cdots,1)\\
  \end{cases}$$
so
$$I+Z^TZ=\left(
  \begin{array}{cccc}
    \langle\eta_{k_1},\eta_{k_1}\rangle & \langle\eta_{k_1},\eta_{k_2}\rangle & \cdots & \langle\eta_{k_1},\eta_{k_{n-m}}\rangle \\
    \langle\eta_{k_2},\eta_{k_1}\rangle & \langle\eta_{k_2},\eta_{k_2}\rangle & \cdots & \langle\eta_{k_2},\eta_{k_{n-m}}\rangle \\
    \hdotsfor{4} \\
    \langle\eta_{k_{n-m}},\eta_{k_1}\rangle & \langle\eta_{k_{n-m}},\eta_{k_2}\rangle & \cdots & \langle\eta_{k_{n-m}},\eta_{k_{n-m}}\rangle \\
  \end{array}
\right)$$

$$R\doteq(I+Z^TZ)^{-1}=\frac{1}{\det(I+Z^TZ)}\left(
  \begin{array}{cccc}
    D_{11} & D_{21} & \cdots & D_{(n-m)1} \\
    D_{12} & D_{22} & \cdots & D_{(n-m)2} \\
    \hdotsfor{4} \\
    D_{1(n-m)} & D_{2(n-m)} & \cdots & D_{(n-m)(n-m)} \\
  \end{array}
\right)$$
where $D_{ij}$ is $(i,j)$ cofactor of $I+Z^TZ$.
\begin{lemma}
On $\{(\widetilde{U}_{i_{1}i_{2}\cdots i_{m}},\widetilde{\phi }_{i_{1}i_{2}\cdots i_{m}})|1\leq i_{1}<i_{2}<\cdots<i_{m}\leq n\}$, the local matrix expression of the Riemannian metric $g$ of $G_{n,m}(\mathbb{R})$ is given by
$$G=\frac{1}{\Delta}\left(\begin{array}{cccc}
    A_{11}(I+Z^TZ)^{-1} & A_{21}(I+Z^TZ)^{-1} & \cdots & A_{m1}(I+Z^TZ)^{-1} \\
    A_{12}(I+Z^TZ)^{-1} & A_{22}(I+Z^TZ)^{-1} & \cdots & A_{m2}(I+Z^TZ)^{-1} \\
    \hdotsfor{4}\\
    A_{1m}(I+Z^TZ)^{-1} & A_{2m}(I+Z^TZ)^{-1} & \cdots & A_{mm}(I+Z^TZ)^{-1} \\
  \end{array}\right)$$
\end{lemma}
\begin{proof}
By definition
$$g={\rm tr}[(I+ZZ^T)^{-1}dZ(I+Z^TZ)^{-1}dZ^T]$$
$$={\rm tr}\left[\frac{1}{\Delta}\left(\begin{array}{cccc}
    A_{11} & A_{21} & \cdots & A_{m1} \\
    A_{12} & A_{22} & \cdots & A_{m2} \\
    \hdotsfor{4}\\
    A_{1m} & A_{2m} & \cdots & A_{mm} \\
  \end{array}\right)\left(\begin{array}{c}
    dX_{1} \\
    dX_{2}\\
    \vdots\\
    dX_{m}\\
  \end{array}\right)(I+Z^TZ)^{-1}\left(\begin{array}{cccc}
    dX^{T}_{1} & dX^{T}_{2} & \cdots & dX^{T}_{m}\\
  \end{array}\right)\right]$$
$$=\frac{1}{\Delta}\left(\begin{array}{cccc}
    dX_{1} & dX_{2} & \cdots & dX_{m}\\
  \end{array}\right)
  \left(\begin{array}{cccc}
    A_{11}R & A_{21}R & \cdots & A_{m1}R \\
    A_{12}R & A_{22}R & \cdots & A_{m2}R \\
    \hdotsfor{4}\\
    A_{1m}R & A_{2m}R & \cdots & A_{mm}R \\
  \end{array}\right)
  \left(\begin{array}{c}
    dX^{T}_{1} \\
    dX^{T}_{2}\\
    \vdots\\
    dX^{T}_{m}\\
  \end{array}\right)$$
where $X_\alpha=row_\alpha Z,\ \alpha=1,2,\cdots,m.$ So we get the result
$$G=\frac{1}{\Delta}\left(\begin{array}{cccc}
    A_{11}(I+Z^TZ)^{-1} & A_{21}(I+Z^TZ)^{-1} & \cdots & A_{m1}(I+Z^TZ)^{-1} \\
    A_{12}(I+Z^TZ)^{-1} & A_{22}(I+Z^TZ)^{-1} & \cdots & A_{m2}(I+Z^TZ)^{-1} \\
    \hdotsfor{4}\\
    A_{1m}(I+Z^TZ)^{-1} & A_{2m}(I+Z^TZ)^{-1} & \cdots & A_{mm}(I+Z^TZ)^{-1} \\
  \end{array}\right)$$
we can see that the matrix $G$ is rank of $m(n-m)$.
\end{proof}

By computation we have
\begin{equation}
G^{-1}=\left(\begin{array}{cccc}
    \langle \xi _1,\xi _1\rangle(I+Z^TZ) & \langle \xi _1,\xi _2\rangle(I+Z^TZ) & \cdots & \langle \xi _1,\xi _m\rangle(I+Z^TZ) \\
    \langle \xi _2,\xi _1\rangle(I+Z^TZ) & \langle \xi _2,\xi _2\rangle(I+Z^TZ) & \cdots & \langle \xi _2,\xi _m\rangle(I+Z^TZ) \\
    \hdotsfor{4}\\
    \langle \xi _m,\xi _1\rangle(I+Z^TZ) & \langle \xi _m,\xi _2\rangle(I+Z^TZ) & \cdots & \langle \xi _m,\xi _m\rangle(I+Z^TZ) \\
  \end{array}\right)\end{equation}
By simple computation, it is obviously that $G$ is the unit matrix of $m(n-m)\times m(n-m)$ at critical point $O_{i_{1}i_{2}\cdots i_{m}}$.

\section{The study of the dynamical system $\dot{x}=-\nabla f$}

We can use the above-mentioned Riemannian metric to define the negative gradient vector field $-\nabla f$ on $G_{n,m}(\mathbb{R})$ for the Morse function $f\doteq\frac{\Delta}{\bar{\Delta}}:G_{n,m}(\mathbb{R})\rightarrow\mathbb{R}$, which has the following local expression on $\{(\widetilde{U}_{i_{1}i_{2}\cdots i_{m}},\widetilde{\phi }_{i_{1}i_{2}\cdots i_{m}})|1\leq i_{1}<i_{2}<\cdots<i_{m}\leq n\}$
$$-\nabla f=-((\nabla f)_{1k_1},(\nabla f)_{1k_2},\cdots,(\nabla f)_{1k_{n-m}},\cdots,(\nabla f)_{mk_1},(\nabla f)_{mk_2},\cdots,(\nabla f)_{mk_{n-m}}).$$
we can get the computation formula of $-\nabla f$ by $g(\nabla f,V)=Vf$,
$$-\nabla f=-g^{ij}\frac{\partial f}{\partial x_{j}}\frac{\partial}{\partial x_{i}}$$
where $g^{ij}$ is the inverse of the matrix expression of the Riemannian metric $g$.

By computation we have
\begin{align*}
(\nabla f)_{ik_s}&=\langle \xi _i,\xi _1\rangle\langle\eta_{k_s},\eta_{k_1}\rangle\frac{\partial f}{\partial x_{1k_{1}}}+\cdots+\langle\xi_i,\xi_1\rangle\langle\eta_{k_s},\eta_{k_{n-m}}\rangle\frac{\partial f}{\partial x_{1k_{n-m}}}\\
&+\langle \xi _i,\xi _2\rangle\langle\eta_{k_s},\eta_{k_1}\rangle\frac{\partial f}{\partial x_{2k_{1}}}+\cdots+\langle\xi_i,\xi_2\rangle\langle\eta_{k_s},\eta_{k_{n-m}}\rangle\frac{\partial f}{\partial x_{2k_{n-m}}}\\
&+\cdots\\
&+\langle\xi_i,\xi_m\rangle\langle\eta_{k_s},\eta_{k_{1}}\rangle\frac{\partial f}{\partial x_{mk_{1}}}+\cdots+\langle \xi _i,\xi _m\rangle\langle\eta_{k_s},\eta_{k_{n-m}}\rangle\frac{\partial f}{\partial x_{mk_{n-m}}}.
\end{align*}

Because (1) and (2) we have
\begin{align*}
\frac{\partial f}{\partial x_{ij}}
&=\frac{1}{\bar{\Delta}^{2}}\left(2(x_{1j}A_{1i}+x_{2j}A_{2i}+\cdots+x_{mj}A_{mi})\cdot\bar{\Delta}-2(\frac{x_{1j}}{\lambda^{2}_{j}}\overline{A}_{1i}+\frac{x_{2j}}{\lambda^{2}_{j}}\overline{A}_{2i}+\cdots+\frac{x_{mj}}{\lambda^{2}_{j}}\overline{A}_{mi})\cdot\Delta\right)
\end{align*}
so we get
\begin{align*}
(\nabla f)_{ik_s}
&=\frac{2}{\bar{\Delta}^{2}}\sum_{t=1}^{n-m}[\langle\xi_i,\xi_1\rangle\langle\eta_{k_s},\eta_{k_{t}}\rangle(x_{1k_{t}}A_{11}+x_{2k_{t}}A_{21}+\cdots+x_{mk_{t}}A_{m1})\cdot\bar{\Delta}\\
&+\langle\xi_i,\xi_2\rangle\langle\eta_{k_s},\eta_{k_{t}}\rangle(x_{1k_{t}}A_{12}+x_{2k_{t}}A_{22}+\cdots+x_{mk_{t}}A_{m2})\cdot\bar{\Delta}+\cdots\\
&+\langle\xi_i,\xi_m\rangle\langle\eta_{k_s},\eta_{k_{t}}\rangle(x_{1k_{t}}A_{1m}+x_{2k_{t}}A_{2m}+\cdots+x_{mk_{t}}A_{mm})\cdot\bar{\Delta}]\\
&-\frac{2}{\bar{\Delta}^{2}}\sum_{t=1}^{n-m}[\langle\xi_i,\xi_1\rangle\langle\eta_{k_s},\eta_{k_{t}}\rangle(\frac{x_{1k_{t}}}{\lambda^{2}_{k_{t}}}\overline{A}_{11}+\frac{x_{2k_{t}}}{\lambda^{2}_{k_{t}}}\overline{A}_{21}+\cdots+\frac{x_{mk_{t}}}{\lambda^{2}_{k_{t}}}\overline{A}_{m1})\cdot\Delta\\
&+\langle\xi_i,\xi_2\rangle\langle\eta_{k_s},\eta_{k_{t}}\rangle(\frac{x_{1k_{t}}}{\lambda^{2}_{k_{t}}}\overline{A}_{12}+\frac{x_{2k_{t}}}{\lambda^{2}_{k_{t}}}\overline{A}_{22}+\cdots+\frac{x_{mk_{t}}}{\lambda^{2}_{k_{t}}}\overline{A}_{m2})\cdot\Delta+\cdots\\
&+\langle\xi_i,\xi_m\rangle\langle\eta_{k_s},\eta_{k_{t}}\rangle(\frac{x_{1k_{t}}}{\lambda^{2}_{k_{t}}}\overline{A}_{1m}+\frac{x_{2k_{t}}}{\lambda^{2}_{k_{t}}}\overline{A}_{2m}+\cdots+\frac{x_{mk_{t}}}{\lambda^{2}_{k_{t}}}\overline{A}_{mm})\cdot\Delta]\\
&=\frac{2}{\bar{\Delta}^{2}}[\langle\xi_i,\xi_1\rangle\cdot\Delta\cdot\bar{\Delta}\cdot x_{1k_{s}}+\langle\xi_i,\xi_2\rangle\cdot\Delta\cdot\bar{\Delta}\cdot x_{2k_{s}}+\cdots+\langle\xi_i,\xi_m\rangle\cdot\Delta\cdot\bar{\Delta}\cdot x_{mk_{s}}]\\
&-\frac{2}{\bar{\Delta}^{2}}[\langle\xi_i,\xi_1\rangle\cdot\Delta\cdot\bar{\Delta}\cdot x_{1k_{s}}+\langle\xi_i,\xi_1\rangle\cdot\Delta\left(\overline{A}_{11}(\frac{1}{\lambda^{2}_{k_{s}}}-\frac{1}{\lambda^{2}_{i_{1}}})x_{1k_{s}}+\cdots+\overline{A}_{m1}(\frac{1}{\lambda^{2}_{k_{s}}}-\frac{1}{\lambda^{2}_{i_{m}}})x_{mk_{s}}\right)\\
&+\langle\xi_i,\xi_2\rangle\cdot\Delta\cdot\bar{\Delta}\cdot x_{2k_{s}}+\langle\xi_i,\xi_2\rangle\cdot\Delta\left(\overline{A}_{12}(\frac{1}{\lambda^{2}_{k_{s}}}-\frac{1}{\lambda^{2}_{i_{1}}})x_{1k_{s}}+\cdots+\overline{A}_{m2}(\frac{1}{\lambda^{2}_{k_{s}}}-\frac{1}{\lambda^{2}_{i_{m}}})x_{mk_{s}}\right)+\cdots\\
&+\langle\xi_i,\xi_m\rangle\cdot\Delta\cdot\bar{\Delta}\cdot x_{mk_{s}}+\langle\xi_i,\xi_m\rangle\cdot\Delta\left(\overline{A}_{1m}(\frac{1}{\lambda^{2}_{k_{s}}}-\frac{1}{\lambda^{2}_{i_{1}}})x_{1k_{s}}+\cdots+\overline{A}_{mm}(\frac{1}{\lambda^{2}_{k_{s}}}-\frac{1}{\lambda^{2}_{i_{m}}})x_{mk_{s}}\right)]\\
&=-\frac{2\Delta}{\bar{\Delta}^{2}}[(\frac{1}{\lambda^{2}_{k_{s}}}-\frac{1}{\lambda^{2}_{i_{1}}})\cdot(\langle\xi_i,\xi_1\rangle\overline{A}_{11}+\langle\xi_i,\xi_2\rangle\overline{A}_{12}+\cdots+\langle\xi_i,\xi_m\rangle\overline{A}_{1m})\cdot x_{1k_{s}}\\
&+(\frac{1}{\lambda^{2}_{k_{s}}}-\frac{1}{\lambda^{2}_{i_{2}}})\cdot(\langle\xi_i,\xi_1\rangle\overline{A}_{21}+\langle\xi_i,\xi_2\rangle\overline{A}_{22}+\cdots+\langle\xi_i,\xi_m\rangle\overline{A}_{2m})\cdot x_{2k_{s}}+\cdots\\
&+(\frac{1}{\lambda^{2}_{k_{s}}}-\frac{1}{\lambda^{2}_{i_{m}}})\cdot(\langle\xi_i,\xi_1\rangle\overline{A}_{m1}+\langle\xi_i,\xi_2\rangle\overline{A}_{m2}+\cdots+\langle\xi_i,\xi_m\rangle\overline{A}_{mm})\cdot x_{mk_{s}}]
\end{align*}
by the computation we have
\begin{equation}
\left(\begin{array}{c}
    (\nabla f)_{1k_1} \\
    (\nabla f)_{1k_2} \\
    \vdots\\
    (\nabla f)_{1k_{n-m}}\\
    \vdots\\
    (\nabla f)_{mk_1}\\
    (\nabla f)_{mk_2}\\
    \vdots\\
    (\nabla f)_{mk_{n-m}}\\
    \end{array}\right)=-\frac{2\Delta}{\bar{\Delta}^{2}}\left(\begin{array}{ccc}
    \langle \xi _1,\xi _1\rangle I_{n-m} &  \cdots & \langle \xi _1,\xi _m\rangle I_{n-m} \\
    \langle \xi _2,\xi _1\rangle I_{n-m} &  \cdots & \langle \xi _2,\xi _m\rangle I_{n-m} \\
    \hdotsfor{3}\\
    \langle \xi _m,\xi _1\rangle I_{n-m} &  \cdots & \langle \xi _m,\xi _m\rangle I_{n-m} \\
  \end{array}\right)
  \cdot F\cdot H\cdot X
\end{equation}
where
$$X=(x_{1k_1},x_{1k_2},\cdots,x_{1k_{n-m}},x_{2k_1},x_{2k_2},\cdots,x_{2k_{n-m}},\cdots,x_{mk_1},x_{mk_2},\cdots,x_{mk_{n-m}})^T.$$
$$F=\left(\begin{array}{cccc}
    \overline{A}_{11}I_{n-m} & \overline{A}_{21}I_{n-m} & \cdots & \overline{A}_{m1}I_{n-m} \\
    \overline{A}_{12}I_{n-m} & \overline{A}_{22}I_{n-m} & \cdots & \overline{A}_{m1}I_{n-m} \\
    \hdotsfor{4}\\
    \overline{A}_{1m}I_{n-m} & \overline{A}_{2m}I_{n-m} & \cdots & \overline{A}_{mm}I_{n-m} \\
  \end{array}\right)$$
$$H=\begin{pmatrix}
H_{11} & & & \text{{\huge{0}}}\\
& H_{22} & & \\
&  & \ddots \\
\text{{\huge{0}}} & & & H_{(n-m)(n-m)}
\end{pmatrix}$$
where
$$H_{ss}=\begin{pmatrix}
(\frac{1}{\lambda^{2}_{k_{s}}}-\frac{1}{\lambda^{2}_{i_{1}}}) & & & \text{{\huge{0}}}\\
& (\frac{1}{\lambda^{2}_{k_{s}}}-\frac{1}{\lambda^{2}_{i_{2}}}) & & \\
&  & \ddots \\
\text{{\huge{0}}} & & & (\frac{1}{\lambda^{2}_{k_{s}}}-\frac{1}{\lambda^{2}_{i_{m}}})
\end{pmatrix}$$
$s=1,2,\cdots,n-m.$

\begin{lemma}
$O_{i_{1}i_{2}\cdots i_{m}}$ is a hyperbolic singular point of $\dot{x}=-\nabla f$, and the linear part of $\dot{x}=-\nabla f$ at $O_{i_{1}i_{2}\cdots i_{m}}$ has the expression $AX$ on $(\widetilde{U}_{i_{1}i_{2}\cdots i_{m}},\widetilde{\phi }_{i_{1}i_{2}\cdots i_{m}})$ where
$$A=-2\lambda^{2}_{i_{1}}\lambda^{2}_{i_{2}}\cdots\lambda^{2}_{i_{m}}\mathrm{diag}\left(1-\frac{\lambda^{2}_{i_{1}}}{\lambda^{2}_{k_{1}}},\cdots,1-\frac{\lambda^{2}_{i_{1}}}{\lambda^{2}_{k_{n-m}}},
\cdots,1-\frac{\lambda^{2}_{i_{m}}}{\lambda^{2}_{k_{1}}},\cdots,1-\frac{\lambda^{2}_{i_{m}}}{\lambda^{2}_{k_{n-m}}}\right)$$
and
$X=(x_{1k_1},x_{1k_2},\cdots,x_{1k_{n-m}},x_{2k_1},x_{2k_2},\cdots,x_{2k_{n-m}},\cdots,x_{mk_1},x_{mk_2},\cdots,x_{mk_{n-m}})^T.$
\end{lemma}
\begin{proof}
Because at $O_{i_{1}i_{2}\cdots i_{m}}$ by comptutation $\frac{2\Delta}{\bar{\Delta}^{2}}(O_{i_{1}i_{2}\cdots i_{m}})=2\lambda^{4}_{i_{1}}\lambda^{4}_{i_{2}}\cdots\lambda^{4}_{i_{m}}$,
$$F(O_{i_{1}i_{2}\cdots i_{m}})=\begin{pmatrix}
(\frac{1}{\lambda^{2}_{i_{2}}}\frac{1}{\lambda^{2}_{i_{3}}}\cdots\frac{1}{\lambda^{2}_{i_{m}}})I_{n-m} & & & \text{{\huge{0}}}\\
& (\frac{1}{\lambda^{2}_{i_{1}}}\frac{1}{\lambda^{2}_{i_{3}}}\cdots\frac{1}{\lambda^{2}_{i_{m}}})I_{n-m} & & \\
&  & \ddots \\
\text{{\huge{0}}} & & & (\frac{1}{\lambda^{2}_{i_{1}}}\frac{1}{\lambda^{2}_{i_{2}}}\cdots\frac{1}{\lambda^{2}_{i_{m-1}}})I_{n-m}
\end{pmatrix}$$
then by (4) we can get $$A=-2\lambda^{2}_{i_{1}}\lambda^{2}_{i_{2}}\cdots\lambda^{2}_{i_{m}}\mathrm{diag}\left(1-\frac{\lambda^{2}_{i_{1}}}{\lambda^{2}_{k_{1}}},\cdots,1-\frac{\lambda^{2}_{i_{1}}}{\lambda^{2}_{k_{n-m}}},
\cdots,1-\frac{\lambda^{2}_{i_{m}}}{\lambda^{2}_{k_{1}}},\cdots,1-\frac{\lambda^{2}_{i_{m}}}{\lambda^{2}_{k_{n-m}}}\right)$$
\end{proof}

Let $M$ be a compact smooth Riemannian manifold with metirc $g$, and $f$ be a Morse function on $M$, then the negative gradient vector field $-\nabla f$ determines a smooth flow $\varphi:\mathbb{R}\times M\rightarrow M$, and $\varphi_{t}$ is a diffeomorphism of $M$ for all $t\in \mathbb{R}$(see [4]).

\begin{definition}[{see [4] }]
Let $p\in M$ be a non-degenerate critical point of Morse function f,
the stable manifold of $p$ is defined to be $$W^{s}(p)=\{ x\in M\mid  \lim_{t\rightarrow+\infty}\varphi_{t}(x)=p  \}$$
the unstable manifold of $p$ is defined to be
$$W^{u}(p)=\{ x\in M\mid  \lim_{t\rightarrow-\infty}\varphi_{t}(x)=p  \}$$

\end{definition}

Obviously
$$\partial_{1k_{1}},\partial_{1k_{2}},\cdots,\partial_{1k_{n-m}},\partial_{2k_{1}},\cdots,\partial_{2k_{n-m}},\partial_{mk_{1}},\cdots,\partial_{mk_{n-m}}$$
is the orthonormal basis of the tangent space on $\widetilde{U}_{i_{1}i_{2}\cdots i_{m}}$, where $\partial_{ij}=\frac{\partial}{\partial x_{ij}}$.

By the computation in lemma 3., we have linear part $A=-H_{O_{i_{1}i_{2}\cdots i_{m}}}(f)$, so use the stabel and unstable manifold theorem for a Morse function(see [4]), we identified the positive definite eigenvalue subspace
$$E^{s}(O_{i_{1}i_{2}\cdots i_{m}})=$$
$$\rm{span}_{\mathbb{R}}\{\partial_{1(i_{1}+1)},\partial_{1(i_{1}+2)},\cdots,\widehat{\partial_{1i_{2}}},\cdots,\widehat{\partial_{1i_{3}}},\cdots,\widehat{\partial_{1i_{m}}},\cdots,\partial_{1n},$$
$$\partial_{2(i_{2}+1)},\partial_{2(i_{2}+2)},\cdots,\widehat{\partial_{2i_{3}}},\cdots,\widehat{\partial_{2i_{m}}},\cdots,\partial_{2n},$$
$$\partial_{3(i_{3}+1)},\partial_{3(i_{3}+1)},\cdots,\widehat{\partial_{3i_{4}}},\cdots,\widehat{\partial_{3i_{m}}},\cdots,\partial_{3n},$$
$$\cdots,$$
$$\partial_{m(i_{m}+1)},\partial_{m(i_{m}+2)},\cdots,\partial_{mn}\}$$
and negative definite eigenvalue subspace
$$E^{u}(O_{i_{1}i_{2}\cdots i_{m}})=$$
$$\rm{span}_{\mathbb{R}}\{\partial_{11},\partial_{12},\cdots,\partial_{1(i_{1}-1)},$$
$$\partial_{21},\partial_{22},\cdots,\widehat{\partial_{2i_{1}}},\cdots,\partial_{2(i_{2}-1)},$$
$$\partial_{31},\partial_{32},\cdots,\widehat{\partial_{3i_{1}}},\cdots,\widehat{\partial_{3i_{2}}},\cdots,\partial_{3(i_{3}-1)},$$
$$\cdots,$$
$$\partial_{m1},\partial_{m2},\cdots,\widehat{\partial_{mi_{1}}},\cdots,\widehat{\partial_{mi_{2}}},\cdots,\widehat{\partial_{mi_{3}}},\cdots,\widehat{\partial_{mi_{m-1}}},\cdots,\partial_{m(i_{m}-1)}\}$$
and
$$E^{s}(O_{i_{1}i_{2}\cdots i_{m}})=T_{O_{i_{1}i_{2}\cdots i_{m}}}W^{s}(O_{i_{1}i_{2}\cdots i_{m}}), \ E^{u}(O_{i_{1}i_{2}\cdots i_{m}})=T_{O_{i_{1}i_{2}\cdots i_{m}}}W^{u}(O_{i_{1}i_{2}\cdots i_{m}}).$$
where $\widehat{\partial_{ij}}$ means this one is empty.

\begin{definition}[{see [6]}]
An invariant manifold $N$ of a vector field $V$ on a manifold $M$ and of the corresponding differential equation $\dot{x}=V(x)$ is defined to be a submanifold of $M$ which is tangent to the vector field $V$ at each of its points.

An invariant manifold $N$ is global if the initial value problem $$\dot{x}=V(x), \ x(0)=p$$
has a global solution $x=x(t)$, $(-\infty<t<+\infty)$ for any $p\in N$.
\end{definition}

\begin{lemma}
The following sets as global invariant manifold of $\dot{x}=-\nabla f$ on $G_{n,m}(\mathbb{R})$,
$$\widetilde{U}_{i_{1}i_{2}\cdots i_{m}}(x_{\beta k_{s_{1}}},x_{\beta k_{s_{2}}},\cdots,x_{\beta k_{s_{t}}})$$
$$=\pi(\{(x_{\alpha k})\in U_{i_{1}i_{2}\cdots i_{m}}\mid x_{\alpha k}=0, (\alpha,k)\neq (\beta,k_{s_{p}}),p=1,\cdots,t \}),$$
$$\widetilde{U}_{i_{1}i_{2}\cdots i_{m}}(x_{\beta_{1} k_{s^{1}_{1}}},x_{\beta_{1} k_{s^{1}_{2}}},\cdots,x_{\beta_{1} k_{s^{1}_{t_{1}}}},x_{\beta_{2} k_{s^{2}_{1}}},x_{\beta_{2} k_{s^{2}_{2}}},\cdots,x_{\beta_{2} k_{s^{2}_{t_{2}}}})$$
$$=\pi(\{(x_{\alpha k})\in U_{i_{1}i_{2}\cdots i_{m}}\mid x_{\alpha k}=0, (\alpha,k)\neq (\beta_{j},k_{s^{j}_{p}}),j=1,2;p=1,\cdots,t_{j} \}),$$
$$\widetilde{U}_{i_{1}i_{2}\cdots i_{m}}(x_{\beta_{1} k_{s^{1}_{1}}},x_{\beta_{1} k_{s^{1}_{2}}},\cdots,x_{\beta_{1} k_{s^{1}_{t_{1}}}},x_{\beta_{2} k_{s^{2}_{1}}},x_{\beta_{2} k_{s^{2}_{2}}},\cdots,x_{\beta_{2} k_{s^{2}_{t_{2}}}},x_{\beta_{3} k_{s^{3}_{1}}},x_{\beta_{3} k_{s^{3}_{2}}},\cdots,x_{\beta_{3} k_{s^{3}_{t_{3}}}})$$
$$=\pi(\{(x_{\alpha k})\in U_{i_{1}i_{2}\cdots i_{m}}\mid x_{\alpha k}=0, (\alpha,k)\neq (\beta_{j},k_{s^{j}_{p}}),j=1,2,3;p=1,\cdots,t_{j} \}),$$
$$\cdots\cdots\cdots$$
$$\widetilde{U}_{i_{1}i_{2}\cdots i_{m}}(x_{\beta_{j} k_{s^{j}_{1}}},x_{\beta_{j} k_{s^{j}_{2}}},\cdots,x_{\beta_{j} k_{s^{j}_{t_{j}}}})$$
$$=\pi(\{(x_{\alpha k})\in U_{i_{1}i_{2}\cdots i_{m}}\mid x_{\alpha k}=0, (\alpha,k)\neq (\beta_{j},k_{s^{j}_{p}}),j=1,2,3,\cdots,m;p=1,\cdots,t_{j} \}).$$
\end{lemma}

\begin{proof}
For any $\widetilde{U}_{i_{1}i_{2}\cdots i_{m}}(\cdots)$, when $x_{\beta k}\not\in(\cdots)$, then $x_{\beta k}=0$, so $\dot{x}_{\beta k}(t)=0$ on $\widetilde{U}_{i_{1}i_{2}\cdots i_{m}}(\cdots)$ and $-(\nabla f)_{\beta k}\mid_{\widetilde{U}_{i_{1}i_{2}\cdots i_{m}}(\cdots)}=\dot{x}_{\beta k}(t)=0$, then $-(\nabla f)\mid_{\widetilde{U}_{i_{1}i_{2}\cdots i_{m}}(\cdots)}$ is tangent vector field on $\widetilde{U}_{i_{1}i_{2}\cdots i_{m}}(\cdots)$. So $\widetilde{U}_{i_{1}i_{2}\cdots i_{m}}(\cdots)$ is invariant manifold of $\dot{x}=-\nabla f$ on $G_{n,m}(\mathbb{R})$.

Because $G_{n,m}(\mathbb{R})$ is compact smooth manifold, so the initial value problem $$\dot{x}=-\nabla f, \ x(0)=p$$ on $G_{n,m}(\mathbb{R})$ has global solution. $$\widetilde{U}_{i_{1}i_{2}\cdots i_{m}}(\cdots)\in G_{n,m}(\mathbb{R})$$ the initial value problem $$\dot{x}=-(\nabla f)\mid_{\widetilde{U}_{i_{1}i_{2}\cdots i_{m}}(\cdots)}, \ x(0)=p, \  p\in\widetilde{U}_{i_{1}i_{2}\cdots i_{m}}(\cdots)$$ on $\widetilde{U}_{i_{1}i_{2}\cdots i_{m}}(\cdots)$ also has global solution. Then $\widetilde{U}_{i_{1}i_{2}\cdots i_{m}}(\cdots)$ is the global invariant manifold.

\end{proof}

\begin{lemma}[{see [6]}]
Let $V$ be a smooth vector field on a manifold $M$ and $p\in M$ be a hyperbolic singular point of $V$. Let $N$ be a global invariant manifold of $V$ in $M$ and $p\in N$. Then we have the following sets equalities
$$W^{s}_{N}(p)=W^{s}(p)\cap N, \ W^{u}_{N}(p)=W^{u}(p)\cap N,$$
where $W^{s}_{N}(p)$,$W^{u}_{N}(p)$ are the stable and unstable manifold of $V\mid_{N}$, the restriction of $V$ on $N$ at $p$, particularly we have
$$W^{s}_{N}(p)=W^{s}(p) \ \ if \ \dim W^{s}_{N}(p)=\dim W^{s}(p)$$
$$W^{u}_{N}(p)=W^{u}(p) \ \ if  \ \dim W^{u}_{N}(p)=\dim W^{s}(p).$$
\end{lemma}

\begin{lemma}
The stable and unstable manifold of $O_{i_{1}i_{2}\cdots i_{m}}$, have the following results
\begin{description}
\item[a)] $$W^{s}(O_{i_{1}i_{2}\cdots i_{m}})\subset\widetilde{U}_{i_{1}i_{2}\cdots i_{m}} (x_{\beta_{j}k_{s^{j}_{1}}},x_{\beta_{j}k_{s^{j}_{2}}},\cdots,x_{\beta_{j} k_{s^{j}_{t_{j}}}})$$
where $j=1,2,\cdots,m$;\\
when $j=1$, let $\beta_{1}=1$, $k_{s^{1}_{p}}=i_{1}+1,i_{1}+2,\cdots,\widehat{i_{2}},\cdots,\widehat{i_{3}},\cdots,\widehat{i_{m}},\cdots,n$;\\
when $j=2$, let $\beta_{2}=2$, $k_{s^{2}_{p}}=i_{2}+1,i_{2}+2,\cdots,\widehat{i_{3}},\cdots,\widehat{i_{4}},\cdots,\widehat{i_{m}},\cdots,n$;\\
$$\cdots\cdots\cdots$$
when $j=m$, let $\beta_{m}=m$, $k_{s^{m}_{p}}=i_{m}+1,i_{m}+2,\cdots,n$.

\item[b)] $$W^{u}(O_{i_{1}i_{2}\cdots i_{m}})\subset\widetilde{U}_{i_{1}i_{2}\cdots i_{m}}
(x_{\beta_{j}k_{s^{j}_{1}}},x_{\beta_{j}k_{s^{j}_{2}}},\cdots,x_{\beta_{j} k_{s^{j}_{t_{j}}}})$$
where $j=1,2,\cdots,m$;\\
when $j=1$, let $\beta_{1}=1$, $k_{s^{1}_{p}}=1,2,\cdots,i_{1}-1$;\\
when $j=2$, let $\beta_{2}=2$, $k_{s^{2}_{p}}=1,2,\cdots,\widehat{i_{1}},\cdots,i_{2}-1$;\\
$$\cdots\cdots\cdots$$
when $j=m$, let $\beta_{m}=m$, $k_{s^{m}_{p}}=1,2,\cdots,\widehat{i_{1}},\cdots,\widehat{i_{2}},\cdots,\widehat{i_{m-1}},\cdots,i_{m}-1$.
\end{description}
\end{lemma}
\begin{proof}
Because $O_{i_{1}i_{2}\cdots i_{m}}\in\widetilde{U}_{i_{1}i_{2}\cdots i_{m}} (x_{\beta_{j}k_{s^{j}_{1}}},x_{\beta_{j}k_{s^{j}_{2}}},\cdots,x_{\beta_{j} k_{s^{j}_{t_{j}}}})$, where $j=1,2,\cdots,m$; so $O_{i_{1}i_{2}\cdots i_{m}}$ is a singular point of $-\nabla f\mid_{\widetilde{U}_{i_{1}i_{2}\cdots i_{m}} (x_{\beta_{j}k_{s^{j}_{1}}},x_{\beta_{j}k_{s^{j}_{2}}},\cdots,x_{\beta_{j} k_{s^{j}_{t_{j}}}})}$.\\
By lemma 6., $\widetilde{U}_{i_{1}i_{2}\cdots i_{m}} (x_{\beta_{j}k_{s^{j}_{1}}},x_{\beta_{j}k_{s^{j}_{2}}},\cdots,x_{\beta_{j} k_{s^{j}_{t_{j}}}})$ is a global invariant manifold, so $O_{i_{1}i_{2}\cdots i_{m}}$ is the hyperbolic singular point of $-\nabla f\mid_{\widetilde{U}_{i_{1}i_{2}\cdots i_{m}} (x_{\beta_{j}k_{s^{j}_{1}}},x_{\beta_{j}k_{s^{j}_{2}}},\cdots,x_{\beta_{j} k_{s^{j}_{t_{j}}}})}$ and the linear part at $O_{i_{1}i_{2}\cdots i_{m}}$ is
$$A=-2\lambda^{2}_{i_{1}}\lambda^{2}_{i_{2}}\cdots\lambda^{2}_{i_{m}}\mathrm{diag}\left(1-\frac{\lambda^{2}_{i_{1}}}{\lambda^{2}_{k^{1}_{1}}},\cdots,1-\frac{\lambda^{2}_{i_{1}}}{\lambda^{2}_{k^{1}_{t_{1}}}},
\cdots,1-\frac{\lambda^{2}_{i_{m}}}{\lambda^{2}_{k^{m}_{1}}},\cdots,1-\frac{\lambda^{2}_{i_{m}}}{\lambda^{2}_{k^{m}_{t_{m}}}}\right)$$
where $k^{1}_{t_{p}}=i_{1}+1,i_{1}+2,\cdots,\widehat{i_{2}},\cdots,\widehat{i_{3}},\cdots,\widehat{i_{m}},\cdots,n$;\\
$k^{2}_{t_{p}}=i_{2}+1,i_{2}+2,\cdots,\widehat{i_{3}},\cdots,\widehat{i_{4}},\cdots,\widehat{i_{m}},\cdots,n$; $\cdots\cdots$;$k^{m}_{t_{p}}=i_{m}+1,i_{m}+2,\cdots,n$.\\
Then the positive definite eigenvalue subspace is
$$E^{s}_{\widetilde{U}_{i_{1}i_{2}\cdots i_{m}} (x_{\beta_{j}k_{s^{j}_{1}}},x_{\beta_{j}k_{s^{j}_{2}}},\cdots,x_{\beta_{j} k_{s^{j}_{t_{j}}}})}(O_{i_{1}i_{2}\cdots i_{m}})=$$
$$\rm{span}_{\mathbb{R}}\{\partial_{1(i_{1}+1)},\partial_{1(i_{1}+2)},\cdots,\widehat{\partial_{1i_{2}}},\cdots,\widehat{\partial_{1i_{3}}},\cdots,\widehat{\partial_{1i_{m}}},\cdots,\partial_{1n},$$
$$\partial_{2(i_{2}+1)},\partial_{2(i_{2}+2)},\cdots,\widehat{\partial_{2i_{3}}},\cdots,\widehat{\partial_{2i_{m}}},\cdots,\partial_{2n},$$
$$\partial_{3(i_{3}+1)},\partial_{3(i_{3}+1)},\cdots,\widehat{\partial_{3i_{4}}},\cdots,\widehat{\partial_{3i_{m}}},\cdots,\partial_{3n},$$
$$\cdots,$$
$$\partial_{m(i_{m}+1)},\partial_{m(i_{m}+2)},\cdots,\partial_{mn}\}.$$
So we have $E^{s}_{\widetilde{U}_{i_{1}i_{2}\cdots i_{m}} (x_{\beta_{j}k_{s^{j}_{1}}},x_{\beta_{j}k_{s^{j}_{2}}},\cdots,x_{\beta_{j} k_{s^{j}_{t_{j}}}})}(O_{i_{1}i_{2}\cdots i_{m}})=E^{s}(O_{i_{1}i_{2}\cdots i_{m}})$, and we get
$$\dim W^{s}_{\widetilde{U}_{i_{1}i_{2}\cdots i_{m}} (x_{\beta_{j}k_{s^{j}_{1}}},x_{\beta_{j}k_{s^{j}_{2}}},\cdots,x_{\beta_{j} k_{s^{j}_{t_{j}}}})}(O_{i_{1}i_{2}\cdots i_{m}})=\dim E^{s}_{\widetilde{U}_{i_{1}i_{2}\cdots i_{m}} (x_{\beta_{j}k_{s^{j}_{1}}},x_{\beta_{j}k_{s^{j}_{2}}},\cdots,x_{\beta_{j} k_{s^{j}_{t_{j}}}})}(O_{i_{1}i_{2}\cdots i_{m}})$$ $$=\dim E^{s}(O_{i_{1}i_{2}\cdots i_{m}})=\dim W^{s}(O_{i_{1}i_{2}\cdots i_{m}})$$
by Lemma 7., we have
$$W^{s}(O_{i_{1}i_{2}\cdots i_{m}})\subset\widetilde{U}_{i_{1}i_{2}\cdots i_{m}} (x_{\beta_{j}k_{s^{j}_{1}}},x_{\beta_{j}k_{s^{j}_{2}}},\cdots,x_{\beta_{j} k_{s^{j}_{t_{j}}}});$$
where $j=1,2,\cdots,m$;\\
when $j=1$, let $\beta_{1}=1$, $k_{s^{1}_{p}}=i_{1}+1,i_{1}+2,\cdots,\widehat{i_{2}},\cdots,\widehat{i_{3}},\cdots,\widehat{i_{m}},\cdots,n$;\\
when $j=2$, let $\beta_{2}=2$, $k_{s^{2}_{p}}=i_{2}+1,i_{2}+2,\cdots,\widehat{i_{3}},\cdots,\widehat{i_{4}},\cdots,\widehat{i_{m}},\cdots,n$;\\
$$\cdots\cdots\cdots$$
when $j=m$, let $\beta_{m}=m$, $k_{s^{m}_{p}}=i_{m}+1,i_{m}+2,\cdots,n$.

By the same way, we can get $$W^{u}(O_{i_{1}i_{2}\cdots i_{m}})\subset\widetilde{U}_{i_{1}i_{2}\cdots i_{m}}
(x_{\beta_{j}k_{s^{j}_{1}}},x_{\beta_{j}k_{s^{j}_{2}}},\cdots,x_{\beta_{j} k_{s^{j}_{t_{j}}}});$$
where $j=1,2,\cdots,m$;\\
when $j=1$, let $\beta_{1}=1$, $k_{s^{1}_{p}}=1,2,\cdots,i_{1}-1$;\\
when $j=2$, let $\beta_{2}=2$, $k_{s^{2}_{p}}=1,2,\cdots,\widehat{i_{1}},\cdots,i_{2}-1$;\\
$$\cdots\cdots\cdots$$
when $j=m$, let $\beta_{m}=m$, $k_{s^{m}_{p}}=1,2,\cdots,\widehat{i_{1}},\cdots,\widehat{i_{2}},\cdots,\widehat{i_{m-1}},\cdots,i_{m}-1$.
\end{proof}

\section{The Morse-Smale transversality condition}
In this section we will to proof $f=\frac{\Delta}{\bar{\Delta}}$ is a Morse-Smale funtion.
For all critical points $O_{i_{1}i_{2}\cdots i_{m}}$, $O_{l_{1}l_{2}\cdots l_{m}}$ of $f=\frac{\Delta}{\bar{\Delta}}$ we need to proof the stable and unstable manifolds of $f$ intersect transversally(see [4] and [9]).
\begin{lemma}
Let $O_{i_{1}i_{2}\cdots i_{m}}$, $O_{l_{1}l_{2}\cdots l_{m}}$ be the different critical points of $f$, and let $l_{k}\geq i_{k}$ for some $k\in 1,2,\cdots,m$, then have the following result
$$W^{u}(O_{i_{1}i_{2}\cdots i_{m}})\cap W^{s}(O_{l_{1}l_{2}\cdots l_{m}})=\emptyset$$
\end{lemma}
\begin{proof}
If $W^{u}(O_{i_{1}i_{2}\cdots i_{m}})\cap W^{s}(O_{l_{1}l_{2}\cdots l_{m}})\neq\emptyset$, there is $p\in W^{u}(O_{i_{1}i_{2}\cdots i_{m}})\cap W^{s}(O_{l_{1}l_{2}\cdots l_{m}})$ and $\exists\varphi(t)(-\infty<t<+\infty)$ is the solution of $\dot{x}=-\nabla f$, with $\varphi(0)=p$ and
$$\lim_{t\rightarrow+\infty}\varphi(t)=O_{l_{1}l_{2}\cdots l_{m}},\ \lim_{t\rightarrow-\infty}\varphi(t)=O_{i_{1}i_{2}\cdots i_{m}}$$
so $\varphi(t)\in W^{u}(O_{i_{1}i_{2}\cdots i_{m}})\cap W^{s}(O_{l_{1}l_{2}\cdots l_{m}})$.\\

Because $O_{l_{1}l_{2}\cdots l_{m}}\in\widetilde{U}_{l_{1}l_{2}\cdots l_{m}}$, so $\exists t_{0}>0$ with $t>t_{0}$,$\varphi(t)\in\widetilde{U}_{l_{1}l_{2}\cdots l_{m}}$. By Lemma 8., we have
$$\varphi(t)\in\widetilde{U}_{l_{1}l_{2}\cdots l_{m}}\cap\widetilde{U}_{i_{1}i_{2}\cdots i_{m}}
(x_{\beta_{j}k_{s^{j}_{1}}},x_{\beta_{j}k_{s^{j}_{2}}},\cdots,x_{\beta_{j} k_{s^{j}_{t_{j}}}}), \ t>t_{0},$$
where $j=1,2,\cdots,m$;\\
when $j=1$, let $\beta_{1}=1$, $k_{s^{1}_{p}}=1,2,\cdots,i_{1}-1$;\\
when $j=2$, let $\beta_{2}=2$, $k_{s^{2}_{p}}=1,2,\cdots,\widehat{i_{1}},\cdots,i_{2}-1$;\\
$$\cdots\cdots\cdots$$
when $j=m$, let $\beta_{m}=m$, $k_{s^{m}_{p}}=1,2,\cdots,\widehat{i_{1}},\cdots,\widehat{i_{2}},\cdots,\widehat{i_{m-1}},\cdots,i_{m}-1$.

But because $l_{k}\geq i_{k}$ for some $k\in 1,2,\cdots,m$, so we have
$$\widetilde{U}_{l_{1}l_{2}\cdots l_{m}}\cap\widetilde{U}_{i_{1}i_{2}\cdots i_{m}}
(x_{\beta_{j}k_{s^{j}_{1}}},x_{\beta_{j}k_{s^{j}_{2}}},\cdots,x_{\beta_{j} k_{s^{j}_{t_{j}}}})=\emptyset,$$
where $j=1,2,\cdots,m$;\\
when $j=1$, let $\beta_{1}=1$, $k_{s^{1}_{p}}=1,2,\cdots,i_{1}-1$;\\
when $j=2$, let $\beta_{2}=2$, $k_{s^{2}_{p}}=1,2,\cdots,\widehat{i_{1}},\cdots,i_{2}-1$;\\
$$\cdots\cdots\cdots$$
when $j=m$, let $\beta_{m}=m$, $k_{s^{m}_{p}}=1,2,\cdots,\widehat{i_{1}},\cdots,\widehat{i_{2}},\cdots,\widehat{i_{m-1}},\cdots,i_{m}-1$.

Then we get $$W^{u}(O_{i_{1}i_{2}\cdots i_{m}})\cap W^{s}(O_{l_{1}l_{2}\cdots l_{m}})=\emptyset$$.
\end{proof}

\begin{lemma}
The function $f=\frac{\Delta}{\bar{\Delta}}$ is Morse-Smale funtion.
\end{lemma}
\begin{proof}
Let $O_{i_{1}i_{2}\cdots i_{m}}$, $O_{l_{1}l_{2}\cdots l_{m}}$ be any two different critical points of $f$, and with
$$W^{u}(O_{i_{1}i_{2}\cdots i_{m}})\cap W^{s}(O_{l_{1}l_{2}\cdots l_{m}})\neq\emptyset.$$
By the lemma 9., we know that $l_{1}<i_{1},l_{2}<i_{2},\cdots,l_{m}<i_{m}$. Let $p\in W^{u}(O_{i_{1}i_{2}\cdots i_{m}})\cap W^{s}(O_{l_{1}l_{2}\cdots l_{m}})$, the tangent space of $W^{u}(O_{i_{1}i_{2}\cdots i_{m}})$ at $p$ is
$$T_{p}W^{u}(O_{i_{1}i_{2}\cdots i_{m}})={\rm{span}}_{\mathbb{R}}\{\partial_{11},\partial_{12},\cdots,\partial_{1(i_{1}-1)},$$
$$\partial_{21},\partial_{22},\cdots,\widehat{\partial_{2i_{1}}},\cdots,\partial_{2(i_{2}-1)},$$
$$\partial_{31},\partial_{32},\cdots,\widehat{\partial_{3i_{1}}},\cdots,\widehat{\partial_{3i_{2}}},\cdots,\partial_{3(i_{3}-1)},$$
$$\cdots,$$
$$\partial_{m1},\partial_{m2},\cdots,\widehat{\partial_{mi_{1}}},\cdots,\widehat{\partial_{mi_{2}}},\cdots,\widehat{\partial_{mi_{3}}},\cdots,\widehat{\partial_{mi_{m-1}}},\cdots,\partial_{m(i_{m}-1)}\};$$
The tangant space of $W^{s}(O_{l_{1}l_{2}\cdots l_{m}})$ at $p$ is
$$T_{p}W^{s}(O_{l_{1}l_{2}\cdots l_{m}})={\rm{span}}_{\mathbb{R}}\{\partial_{1(l_{1}+1)},\partial_{1(l_{1}+2)},\cdots,\widehat{\partial_{1l_{2}}},\cdots,\widehat{\partial_{1l_{3}}},\cdots,\widehat{\partial_{1l_{m}}},\cdots,\partial_{1n},$$
$$\partial_{2(l_{2}+1)},\partial_{2(l_{2}+2)},\cdots,\widehat{\partial_{2l_{3}}},\cdots,\widehat{\partial_{2l_{m}}},\cdots,\partial_{2n},$$
$$\partial_{3(l_{3}+1)},\partial_{3(l_{3}+1)},\cdots,\widehat{\partial_{3l_{4}}},\cdots,\widehat{\partial_{3l_{m}}},\cdots,\partial_{3n},$$
$$\cdots,$$
$$\partial_{m(l_{m}+1)},\partial_{m(l_{m}+2)},\cdots,\partial_{mn}\}.$$
So we have $$T_{p}G_{n,m}(\mathbb{R})\subset T_{p}W^{u}(O_{i_{1}i_{2}\cdots i_{m}})+T_{p}W^{s}(O_{l_{1}l_{2}\cdots l_{m}})$$
and because $$W^{u}(O_{i_{1}i_{2}\cdots i_{m}})\subset G_{n,m}(\mathbb{R}),\ W^{s}(O_{l_{1}l_{2}\cdots l_{m}})\subset G_{n,m}(\mathbb{R})$$
so $$T_{p}G_{n,m}(\mathbb{R})\supset T_{p}W^{u}(O_{i_{1}i_{2}\cdots i_{m}})+T_{p}W^{s}(O_{l_{1}l_{2}\cdots l_{m}})$$
Then we get
$$T_{p}G_{n,m}(\mathbb{R})=T_{p}W^{u}(O_{i_{1}i_{2}\cdots i_{m}})+T_{p}W^{s}(O_{l_{1}l_{2}\cdots l_{m}}).$$
It means that the stable and unstable manifolds of $f$ intersect transversally, $f=\frac{\Delta}{\bar{\Delta}}$ is a Morse-Smale funtion.

\end{proof}

\section{The trajectories connect the critical points}
\begin{theorem}
Let $O_{i_{1}i_{2}\cdots i_{m}}$, $O_{l_{1}l_{2}\cdots l_{m}}$ be any two different critical points of $f$, with
$${\rm{ind}}(O_{i_{1}i_{2}\cdots i_{m}})-{\rm{ind}}(O_{l_{1}l_{2}\cdots l_{m}})=1,$$
and $i_{k}-l_{k}>1$ for some $k\in 1,2,\cdots,m$; then $$W^{u}(O_{i_{1}i_{2}\cdots i_{m}})\cap W^{s}(O_{l_{1}l_{2}\cdots l_{m}})=\emptyset$$
\end{theorem}
\begin{proof}
If $W^{u}(O_{i_{1}i_{2}\cdots i_{m}})\cap W^{s}(O_{l_{1}l_{2}\cdots l_{m}})\neq\emptyset$, there is $p\in W^{u}(O_{i_{1}i_{2}\cdots i_{m}})\cap W^{s}(O_{l_{1}l_{2}\cdots l_{m}})$ and $\exists\varphi(t)(-\infty<t<+\infty)$ is the solution of $\dot{x}=-\nabla f$, with $\varphi(0)=p$ and
$$\lim_{t\rightarrow+\infty}\varphi(t)=O_{l_{1}l_{2}\cdots l_{m}},\ \lim_{t\rightarrow-\infty}\varphi(t)=O_{i_{1}i_{2}\cdots i_{m}}$$
so $\varphi(t)\in W^{u}(O_{i_{1}i_{2}\cdots i_{m}})\cap W^{s}(O_{l_{1}l_{2}\cdots l_{m}})$.\\

Because $O_{l_{1}l_{2}\cdots l_{m}}\in\widetilde{U}_{l_{1}l_{2}\cdots l_{m}}$, so $\exists t_{0}>0$ with $t>t_{0}$,$\varphi(t)\in\widetilde{U}_{l_{1}l_{2}\cdots l_{m}}$. By Lemma 8., we have
$$\varphi(t)\in\widetilde{U}_{l_{1}l_{2}\cdots l_{m}}\cap\widetilde{U}_{i_{1}i_{2}\cdots i_{m}}
(x_{\beta_{j}k_{s^{j}_{1}}},x_{\beta_{j}k_{s^{j}_{2}}},\cdots,x_{\beta_{j} k_{s^{j}_{t_{j}}}}), \ t>t_{0},$$
where $j=1,2,\cdots,m$;\\
when $j=1$, let $\beta_{1}=1$, $k_{s^{1}_{p}}=1,2,\cdots,i_{1}-1$;\\
when $j=2$, let $\beta_{2}=2$, $k_{s^{2}_{p}}=1,2,\cdots,\widehat{i_{1}},\cdots,i_{2}-1$;\\
$$\cdots\cdots\cdots$$
when $j=m$, let $\beta_{m}=m$, $k_{s^{m}_{p}}=1,2,\cdots,\widehat{i_{1}},\cdots,\widehat{i_{2}},\cdots,\widehat{i_{m-1}},\cdots,i_{m}-1$.\\
Because ${\rm{ind}}(O_{i_{1}i_{2}\cdots i_{m}})-{\rm{ind}}(O_{l_{1}l_{2}\cdots l_{m}})=1$, so we have $i_{1}+i_{2}+\cdots+i_{m}-(l_{1}+l_{2}+\cdots+l_{m})=1$, and by $i_{k}-l_{k}>1$ for some $k\in 1,2,\cdots,m$, we get there is a $j\neq k$, $j\in 1,2,\cdots,m$ with $l_{j}>i_{j}$. By it, we have
$$\widetilde{U}_{l_{1}l_{2}\cdots l_{m}}\cap\widetilde{U}_{i_{1}i_{2}\cdots i_{m}}
(x_{\beta_{j}k_{s^{j}_{1}}},x_{\beta_{j}k_{s^{j}_{2}}},\cdots,x_{\beta_{j} k_{s^{j}_{t_{j}}}})=\emptyset,$$
where $j=1,2,\cdots,m$;\\
when $j=1$, let $\beta_{1}=1$, $k_{s^{1}_{p}}=1,2,\cdots,i_{1}-1$;\\
when $j=2$, let $\beta_{2}=2$, $k_{s^{2}_{p}}=1,2,\cdots,\widehat{i_{1}},\cdots,i_{2}-1$;\\
$$\cdots\cdots\cdots$$
when $j=m$, let $\beta_{m}=m$, $k_{s^{m}_{p}}=1,2,\cdots,\widehat{i_{1}},\cdots,\widehat{i_{2}},\cdots,\widehat{i_{m-1}},\cdots,i_{m}-1$.
which contradicts to
$$\varphi(t)\in\widetilde{U}_{l_{1}l_{2}\cdots l_{m}}\cap\widetilde{U}_{i_{1}i_{2}\cdots i_{m}}
(x_{\beta_{j}k_{s^{j}_{1}}},x_{\beta_{j}k_{s^{j}_{2}}},\cdots,x_{\beta_{j} k_{s^{j}_{t_{j}}}}), \ t>t_{0}.$$
So we get $$W^{u}(O_{i_{1}i_{2}\cdots i_{m}})\cap W^{s}(O_{l_{1}l_{2}\cdots l_{m}})=\emptyset.$$

\end{proof}

By Theorem 1., we know that if ${\rm{ind}}(O_{i_{1}i_{2}\cdots i_{m}})-{\rm{ind}}(O_{l_{1}l_{2}\cdots l_{m}})=1$ and $W^{u}(O_{i_{1}i_{2}\cdots i_{m}})\cap W^{s}(O_{l_{1}l_{2}\cdots l_{m}})\neq\emptyset$, so $\exists k$ with $i_{k}-l_{k}=1$, where $k\in 1,2,\cdots,m$.

\begin{theorem}
Let $O_{i_{1}i_{2}\cdots i_{m}}$, $O_{l_{1}l_{2}\cdots l_{m}}$ be any two different critical points of $f$, with
$${\rm{ind}}(O_{i_{1}i_{2}\cdots i_{m}})-{\rm{ind}}(O_{l_{1}l_{2}\cdots l_{m}})=1,$$
Then
$$W^{u}(O_{i_{1}i_{2}\cdots i_{m}})\cap W^{s}(O_{i_{1}i_{2}\cdots (i_{k}-1)\cdots i_{m}})=\widetilde{U}_{i_{1}i_{2}\cdots i_{m}}(x_{k(i_{k}-1)};x_{k(i_{k}-1)}\neq 0)$$

\end{theorem}
\begin{proof}
By Lemma 8., $\forall p\in W^{u}(O_{i_{1}i_{2}\cdots i_{m}})\cap W^{s}(O_{i_{1}i_{2}\cdots (i_{k}-1)\cdots i_{m}})$, the coordinate in $\widetilde{U}_{i_{1}i_{2}\cdots i_{m}}$ is
$$(x_{11}(p),x_{12}(p),\cdots,x_{1(i_{1}-1)}(p),1,$$
$$x_{21}(p),x_{22}(p),\cdots,\widehat{x_{2i_{1}}(p)},\cdots,x_{2(i_{2}-1)}(p),1,$$
$$x_{31}(p),x_{32}(p),\cdots,\widehat{x_{3i_{1}}(p)},\cdots,\widehat{x_{3i_{2}}(p)},\cdots,x_{3(i_{3}-1)}(p),1,$$
$$\cdots,$$
$$x_{k1}(p),x_{k2}(p),\cdots,\widehat{x_{ki_{1}}(p)},\cdots,\widehat{x_{ki_{2}}(p)},\cdots,\widehat{x_{ki_{3}}(p)},\cdots,\widehat{x_{ki_{k-1}}(p)},\cdots,x_{k(i_{k}-1)}(p),1)$$
$$\cdots,$$
$$x_{m1}(p),x_{m2}(p),\cdots,\widehat{x_{mi_{1}}(p)},\cdots,\widehat{x_{mi_{2}}(p)},\cdots,\widehat{x_{mi_{3}}(p)},\cdots,\widehat{x_{mi_{m-1}}(p)},\cdots,x_{m(i_{m}-1)}(p),1);$$
the coordinate in $\widetilde{U}_{i_{1}i_{2}\cdots (i_{k}-1)\cdots i_{m}}$ is
$$(1,y_{1(i_{1}+1)}(p),y_{1(i_{1}+2)}(p),\cdots,\widehat{y_{1i_{2}}(p)},\cdots,\widehat{y_{1i_{3}}(p)},\cdots,\widehat{y_{1i_{m}}(p)},\cdots,y_{1n}(p),$$
$$1,y_{2(i_{2}+1)}(p),y_{2(i_{2}+2)}(p),\cdots,\widehat{y_{2i_{3}}(p)},\cdots,\widehat{y_{2i_{m}}(p)},\cdots,y_{2n}(p),$$
$$1,y_{3(i_{3}+1)}(p),y_{3(i_{3}+1)}(p),\cdots,\widehat{y_{3i_{4}}(p)},\cdots,\widehat{y_{3i_{m}}(p)},\cdots,y_{3n}(p),$$
$$\cdots,$$
$$1,y_{ki_{k}}(p),y_{k(i_{k}+1)}(p),\cdots,\widehat{y_{ki_{k+1}}(p)},\cdots,\widehat{y_{ki_{m}}(p)},\cdots,y_{kn}(p),$$
$$\cdots,$$
$$1,y_{m(i_{m}+1)}(p),y_{m(i_{m}+2)}(p),\cdots,y_{mn}(p)).$$
So the change of the coordinate is $$x_{k(i_{k}-1)}(p)=\frac{1}{y_{ki_{k}}(p)}\neq 0, \ 1=\frac{y_{ki_{k}}(p)}{y_{ki_{k}}(p)};$$
Then we get $$W^{u}(O_{i_{1}i_{2}\cdots i_{m}})\cap W^{s}(O_{i_{1}i_{2}\cdots (i_{k}-1)\cdots i_{m}})=\widetilde{U}_{i_{1}i_{2}\cdots i_{m}}(x_{k(i_{k}-1)};x_{k(i_{k}-1)}\neq 0)$$

\end{proof}

By the theorem 2., we can get the following corollery
\begin{corollery}
Let $O_{i_{1}i_{2}\cdots i_{m}}$, $O_{l_{1}l_{2}\cdots l_{m}}$ be any two different critical points of $f$, with
$${\rm{ind}}(O_{i_{1}i_{2}\cdots i_{m}})-{\rm{ind}}(O_{l_{1}l_{2}\cdots l_{m}})=1.$$
Let $\Gamma(O_{i_{1}i_{2}\cdots i_{m}},O_{l_{1}l_{2}\cdots l_{m}})$ be the numbers of trajectories connectting $O_{i_{1}i_{2}\cdots i_{m}}$ to $O_{l_{1}l_{2}\cdots l_{m}}$, then
$$\Gamma(O_{i_{1}i_{2}\cdots i_{m}},O_{l_{1}l_{2}\cdots l_{m}})=
\begin{cases}
2, \ i_{1}=l_{1},i_{2}=l_{2},\cdots,i_{k-1}=l_{k-1},i_{k}=l_{k}+1,i_{k+1}=l_{k+1},\cdots,i_{m}=l_{m};\\
0, \ otherwise.
\end{cases}
$$
where $k\in 1,2,\cdots,m.$
\end{corollery}

\section{The homology groups of $G_{n,m}(\mathbb{R})$}
Now we choose an orientation of the vector space $E^{u}(O_{i_{1}i_{2}\cdots i_{m}})=T_{O_{i_{1}i_{2}\cdots i_{m}}}W^{u}(O_{i_{1}i_{2}\cdots i_{m}})$ for every critical point $O_{i_{1}i_{2}\cdots i_{m}}$ and denote by $\langle O_{i_{1}i_{2}\cdots i_{m}}\rangle$ the pair consisting of the critical point $O_{i_{1}i_{2}\cdots i_{m}}$ and the orientation.

$$\langle O_{i_{1}i_{2}\cdots i_{m}}\rangle=\{O_{i_{1}i_{2}\cdots i_{m}};(\partial_{11},\partial_{12},\cdots,\partial_{1(i_{1}-1)},$$
$$\partial_{21},\partial_{22},\cdots,\widehat{\partial_{2i_{1}}},\cdots,\partial_{2(i_{2}-1)},$$
$$\partial_{31},\partial_{32},\cdots,\widehat{\partial_{3i_{1}}},\cdots,\widehat{\partial_{3i_{2}}},\cdots,\partial_{3(i_{3}-1)},$$
$$\cdots,$$
$$\partial_{m1},\partial_{m2},\cdots,\widehat{\partial_{mi_{1}}},\cdots,\widehat{\partial_{mi_{2}}},\cdots,\widehat{\partial_{mi_{3}}},\cdots,\widehat{\partial_{mi_{m-1}}},\cdots,\partial_{m(i_{m}-1)})\}$$

By theorem 2., between the critical points $O_{i_{1}i_{2}\cdots i_{m}}$ and $O_{i_{1}i_{2}\cdots (i_{k}-1)\cdots i_{m}}$, we have two orbits $\varphi^{+}(t)$ and $\varphi^{-}(t)$ as following in the coordinate system $\widetilde{U}_{i_{1}i_{2}\cdots i_{m}}$
$$\varphi^{+}(t)=(0,0,0,\cdots,0,x_{k(i_{k}-1)}(t),0,\cdots,0,0,0),\ x_{k(i_{k}-1)}(t)>0$$
where $\lim_{t\rightarrow -\infty}x_{k(i_{k}-1)}(t)=0$, $\lim_{t\rightarrow -\infty}\varphi^{+}(t)=O_{i_{1}i_{2}\cdots i_{m}}$;
$$\varphi^{-}(t)=(0,0,0,\cdots,0,x_{k(i_{k}-1)}(t),0,\cdots,0,0,0),\ x_{k(i_{k}-1)}(t)<0$$
where $\lim_{t\rightarrow -\infty}x_{k(i_{k}-1)}(t)=0$, $\lim_{t\rightarrow -\infty}\varphi^{-}(t)=O_{i_{1}i_{2}\cdots i_{m}}$;
and in the coordinate system $\widetilde{U}_{i_{1}i_{2}\cdots (i_{k}-1)\cdots i_{m}}$
$$\varphi^{+}(t)=(0,0,0,\cdots,0,\frac{1}{x_{k(i_{k}-1)}(t)},0,\cdots,0,0,0),\ x_{k(i_{k}-1)}(t)>0$$
where $\lim_{t\rightarrow +\infty}x_{k(i_{k}-1)}(t)=+\infty$, $\lim_{t\rightarrow +\infty}\varphi^{+}(t)=O_{i_{1}i_{2}\cdots (i_{k}-1)\cdots i_{m}}$;
$$\varphi^{-}(t)=(0,0,0,\cdots,0,\frac{1}{x_{k(i_{k}-1)}(t)},0,\cdots,0,0,0),\ x_{k(i_{k}-1)}(t)<0$$
where $\lim_{t\rightarrow +\infty}x_{k(i_{k}-1)}(t)=-\infty$, $\lim_{t\rightarrow +\infty}\varphi^{+}(t)=O_{i_{1}i_{2}\cdots (i_{k}-1)\cdots i_{m}}$.

\begin{theorem}
$$n_{\varphi^{+}(O_{i_{1}i_{2}\cdots i_{m}},O_{i_{1}i_{2}\cdots (i_{k}-1)\cdots i_{m}})}=1,$$
$$n_{\varphi^{-}(O_{i_{1}i_{2}\cdots i_{m}},O_{i_{1}i_{2}\cdots (i_{k}-1)\cdots i_{m}})}=(-1)^{i_{k}-k}. \ (1\leq k\leq m)$$
\end{theorem}
\begin{proof}
Because the tangent vector of $\varphi^{\pm}(O_{i_{1}i_{2}\cdots i_{m}},O_{i_{1}i_{2}\cdots (i_{k}-1)\cdots i_{m}})$ at $O_{i_{1}i_{2}\cdots i_{m}}$ is $\pm\partial_{k(i_{k}-1)}$, then $\langle O_{i_{1}i_{2}\cdots i_{m}}\rangle$ induces an orientation on the orthogonal complement $E^{u}_{\varphi^{\pm}}(O_{i_{1}i_{2}\cdots i_{m}})$ of $\pm\partial_{k(i_{k}-1)}$ in $E^{u}(O_{i_{1}i_{2}\cdots i_{m}})$ as follow
$$(\pm\partial_{11},\partial_{12},\cdots,\partial_{1(i_{1}-1)},$$
$$\partial_{21},\partial_{22},\cdots,\widehat{\partial_{2i_{1}}},\cdots,\partial_{2(i_{2}-1)},$$
$$\partial_{31},\partial_{32},\cdots,\widehat{\partial_{3i_{1}}},\cdots,\widehat{\partial_{3i_{2}}},\cdots,\partial_{3(i_{3}-1)},$$
$$\cdots,$$
$$\partial_{k1},\partial_{k2},\cdots,\widehat{\partial_{ki_{1}}},\cdots,\widehat{\partial_{ki_{2}}},\cdots,\partial_{k(i_{k}-2)},$$
$$\cdots,$$
$$\partial_{m1},\partial_{m2},\cdots,\widehat{\partial_{mi_{1}}},\cdots,\widehat{\partial_{mi_{2}}},\cdots,\widehat{\partial_{mi_{3}}},\cdots,\widehat{\partial_{mi_{m-1}}},\cdots,\partial_{m(i_{m}-1)}).$$\\
Because
$$\langle O_{i_{1}i_{2}\cdots (i_{k}-1)\cdots i_{m}}\rangle=\{O_{i_{1}i_{2}\cdots (i_{k}-1)\cdots i_{m}};(\partial_{11},\partial_{12},\cdots,\partial_{1(i_{1}-1)},$$
$$\partial_{21},\partial_{22},\cdots,\widehat{\partial_{2i_{1}}},\cdots,\partial_{2(i_{2}-1)},$$
$$\partial_{31},\partial_{32},\cdots,\widehat{\partial_{3i_{1}}},\cdots,\widehat{\partial_{3i_{2}}},\cdots,\partial_{3(i_{3}-1)},$$
$$\cdots,$$
$$\partial_{k1},\partial_{k2},\cdots,\widehat{\partial_{ki_{1}}},\cdots,\widehat{\partial_{ki_{2}}},\cdots,\partial_{k(i_{k}-2)},$$
$$\cdots,$$
$$\partial_{m1},\partial_{m2},\cdots,\widehat{\partial_{mi_{1}}},\cdots,\widehat{\partial_{mi_{2}}},\cdots,\widehat{\partial_{mi_{3}}},\cdots,\widehat{\partial_{mi_{m-1}}},\cdots,\partial_{m(i_{m}-1)})\}$$
to compute $n_{\varphi^{\pm}(O_{i_{1}i_{2}\cdots i_{m}},O_{i_{1}i_{2}\cdots (i_{k}-1)\cdots i_{m}})}$, we compute the Jacobian of the coordinate transformation from $\widetilde{U}_{i_{1}i_{2}\cdots i_{m}}$ to $\widetilde{U}_{i_{1}i_{2}\cdots (i_{k}-1)\cdots i_{m}}$, we denote it by $J_{n_{\varphi^{\pm}(O_{i_{1}i_{2}\cdots i_{m}},O_{i_{1}i_{2}\cdots (i_{k}-1)\cdots i_{m}})}}$,
$$J_{n_{\varphi^{\pm}(O_{i_{1}i_{2}\cdots i_{m}},O_{i_{1}i_{2}\cdots (i_{k}-1)\cdots i_{m}})}}=$$
$${\rm{diag}}(1,\cdots,1,\frac{1}{x_{k(i_{k}-1)}},\cdots,\frac{1}{x_{k(i_{k}-1)}},-\frac{1}{x^{2}_{k(i_{k}-1)}},\frac{1}{x_{k(i_{k}-1)}},\cdots,\frac{1}{x_{k(i_{k}-1)}},1,\cdots,1)$$
$$J_{n_{\varphi^{\pm}(O_{i_{1}i_{2}\cdots i_{m}},O_{i_{1}i_{2}\cdots (i_{k}-1)\cdots i_{m}})}}(\pm\partial_{11},\partial_{12},\cdots,\partial_{1(i_{1}-1)},$$
$$\partial_{21},\partial_{22},\cdots,\widehat{\partial_{2i_{1}}},\cdots,\partial_{2(i_{2}-1)},$$
$$\partial_{31},\partial_{32},\cdots,\widehat{\partial_{3i_{1}}},\cdots,\widehat{\partial_{3i_{2}}},\cdots,\partial_{3(i_{3}-1)},$$
$$\cdots,$$
$$\partial_{k1},\partial_{k2},\cdots,\widehat{\partial_{ki_{1}}},\cdots,\widehat{\partial_{ki_{2}}},\cdots,\partial_{k(i_{k}-2)},$$
$$\cdots,$$
$$\partial_{m1},\partial_{m2},\cdots,\widehat{\partial_{mi_{1}}},\cdots,\widehat{\partial_{mi_{2}}},\cdots,\widehat{\partial_{mi_{3}}},\cdots,\widehat{\partial_{mi_{m-1}}},\cdots,\partial_{m(i_{m}-1)})$$
$$=(\pm\partial_{11},\partial_{12},\cdots,\partial_{1(i_{1}-1)},$$
$$\partial_{21},\partial_{22},\cdots,\widehat{\partial_{2i_{1}}},\cdots,\partial_{2(i_{2}-1)},$$
$$\partial_{31},\partial_{32},\cdots,\widehat{\partial_{3i_{1}}},\cdots,\widehat{\partial_{3i_{2}}},\cdots,\partial_{3(i_{3}-1)},$$
$$\cdots,$$
$$\frac{\partial_{k1}}{x_{k(i_{k}-1)}},\frac{\partial_{k2}}{x_{k(i_{k}-1)}},\cdots,\widehat{\partial_{ki_{1}}},\cdots,\widehat{\partial_{ki_{2}}},\cdots,\frac{\partial_{k(i_{k}-2)}}{x_{k(i_{k}-1)}},$$
$$\cdots,$$
$$\partial_{m1},\partial_{m2},\cdots,\widehat{\partial_{mi_{1}}},\cdots,\widehat{\partial_{mi_{2}}},\cdots,\widehat{\partial_{mi_{3}}},\cdots,\widehat{\partial_{mi_{m-1}}},\cdots,\partial_{m(i_{m}-1)})$$

So we get, if $x_{k(i_{k}-1)}>0$ then $n_{\varphi^{+}(O_{i_{1}i_{2}\cdots i_{m}},O_{i_{1}i_{2}\cdots (i_{k}-1)\cdots i_{m}})}=1$;
if $x_{k(i_{k}-1)}<0$ then $n_{\varphi^{-}(O_{i_{1}i_{2}\cdots i_{m}},O_{i_{1}i_{2}\cdots (i_{k}-1)\cdots i_{m}})}=(-1)^{i_{k}-k}$.
\end{proof}
By Theorem 3., we have $$n(O_{i_{1}i_{2}\cdots i_{m}},O_{i_{1}i_{2}\cdots (i_{k}-1)\cdots i_{m}})=n_{\varphi^{+}(O_{i_{1}i_{2}\cdots i_{m}},O_{i_{1}i_{2}\cdots (i_{k}-1)\cdots i_{m}})}+n_{\varphi^{-}(O_{i_{1}i_{2}\cdots i_{m}},O_{i_{1}i_{2}\cdots (i_{k}-1)\cdots i_{m}})}$$

Then we can get the result about Witten's boundary operator
\begin{lemma}
$$\partial \langle O_{i_{1}i_{2}\cdots i_{m}}\rangle=\sum_{k} (1+(-1)^{i_{k}-k})\langle O_{i_{1}i_{2}\cdots (i_{k}-1)\cdots i_{m}}\rangle, \ (1\leq k\leq m)$$
and if $i_{1}-1,i_{2}-2,\cdots,i_{m}-m$ all is odd, then
$$\partial \langle O_{i_{1}i_{2}\cdots i_{m}}\rangle=0$$
\end{lemma}
\begin{proof}
Because $n_{\varphi^{+}(O_{i_{1}i_{2}\cdots i_{m}},O_{i_{1}i_{2}\cdots (i_{k}-1)\cdots i_{m}})}+n_{\varphi^{-}(O_{i_{1}i_{2}\cdots i_{m}},O_{i_{1}i_{2}\cdots (i_{k}-1)\cdots i_{m}})}=1+(-1)^{i_{k}-k}$ so
$n(O_{i_{1}i_{2}\cdots i_{m}},O_{i_{1}i_{2}\cdots (i_{k}-1)\cdots i_{m}})=1+(-1)^{i_{k}-k}$, and by the definition of Witten's boundary operator, we get
$$\partial \langle O_{i_{1}i_{2}\cdots i_{m}}\rangle=\sum_{k} (1+(-1)^{i_{k}-k})\langle O_{i_{1}i_{2}\cdots (i_{k}-1)\cdots i_{m}}\rangle.$$
When $i_{1}-1,i_{2}-2,\cdots,i_{m}-m$ all is odd, then $1+(-1)^{i_{k}-k}=0$ so get the result.
\end{proof}

Now we can give the most important result in this article,
\begin{theorem}
The homology groups of real Grassmann manifold $G_{n,m}(\mathbb{R})$ with integral coefficients by Witten complex is
$$H_{r}(G_{n,m}(\mathbb{R}),\mathbb{Z})=\bigoplus_{k}(\bigoplus_{j_{1}+j_{2}+\cdots+j_{m}=r\atop j_{k}+1-k=odd}\mathbb{Z}_{2}\oplus\bigoplus_{j_{1}+j_{2}+\cdots+j_{m}=r\atop j_{k}+1-k=even}\mathbb{Z})$$
where $k=1,2,\cdots,m.$
\end{theorem}
\begin{proof}
Since Witten chain complex $$0\rightarrow C_{m(n-m)}\rightarrow C_{m(n-m)-1}\rightarrow\cdots\rightarrow C_{1}\rightarrow C_{0}\rightarrow 0$$
where
$$C_{r}=\bigoplus_{i_{1}+i_{2}+\cdots+i_{m}=r}\mathbb{Z}\langle O_{i_{1}i_{2}\cdots i_{m}}\rangle.$$
By lemma 11., we have
$$\partial(C_{r+1})=\bigoplus_{i_{1}+i_{2}+\cdots+i_{m}=r+1}\mathbb{Z}\partial\langle O_{i_{1}i_{2}\cdots i_{m}}\rangle$$
$$=\bigoplus_{i_{1}+i_{2}+\cdots+i_{m}=r+1}\mathbb{Z}(\sum_{k} (1+(-1)^{i_{k}-k})\langle O_{i_{1}i_{2}\cdots (i_{k}-1)\cdots i_{m}}\rangle);$$
$\ker\{\partial:C_{r}\longrightarrow C_{r-1}\}=\bigoplus_{j_{1}+j_{2}+\cdots+j_{m}=r}\mathbb{Z}\langle O_{j_{1}j_{2}\cdots j_{m}}\rangle$, where $j_{1}-1,j_{2}-2,\cdots,j_{m}-m$ all is odd.
By the definition of homology groups $$H^{W}_{r}(M,\mathbb{Z})=\frac{\ker\{\partial:C_{r}\longrightarrow C_{r-1} \}}{\partial(C_{r+1})},$$
then
$$H^{W}_{r}(G_{n,m}(\mathbb{R}),\mathbb{Z})=\bigoplus_{k}(\bigoplus_{j_{1}+j_{2}+\cdots+j_{m}=r\atop j_{k}+1-k=odd}\mathbb{Z}_{2}\oplus\bigoplus_{j_{1}+j_{2}+\cdots+j_{m}=r\atop j_{k}+1-k=even}\mathbb{Z})$$
where $k=1,2,\cdots,m.$
And because $H^{W}_{r}(G_{n,m}(\mathbb{R}),\mathbb{Z})=H_{r}(G_{n,m}(\mathbb{R}),\mathbb{Z})$, so we get the result.

\end{proof}

\end{CJK}
\end{document}